\documentclass[11pt, reqno]{amsart}

\usepackage{enumitem}

\usepackage{geometry}
\usepackage{color}
\usepackage{amssymb}
\usepackage{amsmath}
\usepackage{arydshln}
\usepackage{hyperref}
\usepackage{bbm}
\usepackage{mathrsfs}
\usepackage{amsthm}
\usepackage{enumitem}

\usepackage{soul}

\usepackage[T1]{fontenc}
\usepackage{dsfont} 
\usepackage{tikz}
\usetikzlibrary{matrix}

\def\BibTeX{{\rm B\kern-.05em{\sc i\kern-.025em b}\kern-.08em
    T\kern-.1667em\lower.7ex\hbox{E}\kern-.125emX}}

\numberwithin{equation}{section}

\DeclareMathOperator*{\argmax}{arg\,max}

\newcommand{\R}{\mathbb{R}}

\newcommand{\1}{\mathbbm{1}}

\newtheorem{Theorem}{Theorem}[section]

\newtheorem{Corollary}[Theorem]{Corollary}
\newtheorem{Lemma}[Theorem]{Lemma}
\newtheorem{Remark}[Theorem]{Remark}
\newtheorem{Definition}[Theorem]{Definition}
\newtheorem{Example}[Theorem]{Example}
\newtheorem{Assumption}{Assumption}
\newtheorem{Construction}[Theorem]{Construction}

\begin{document}

\title[Copulas in arbitrary dimension]{Copula measures and Sklar's Theorem \\in Arbitrary Dimensions}

\author{Fred Espen Benth}
\address[Fred Espen Benth]{Department of Mathematics\\ University of Oslo, Norway}
\email[Fred Espen Benth]{fredb@math.uio.no}

\author{Giulia Di Nunno}
\address[Giulia Di Nunno]{Department of Mathematics\\ University of Oslo, Norway}
\email[Giulia Di Nunno]{giulian@math.uio.no}

\author{Dennis Schroers}
\address[Dennis Schroers]{Department of Mathematics\\ University of Oslo, Norway}
\email[Dennis Schroers]{dennissc@math.uio.no}

\date{\today}



\maketitle

\begin{abstract}
Although copulas are used and defined for various infinite-dimensional objects (e.g. Gaussian processes and Markov processes), there is no prevalent notion of a copula
that unifies these concepts.
We propose a unified approach and define copulas as probability measures on general product spaces. For this we prove Sklar's Theorem in this infinite-dimensional setting. We show how to transfer this result to various function space settings and describe how to model and approximate dependent probability measures in these spaces in the realm of copulas. 
\end{abstract}
\allowdisplaybreaks
\section{Introduction}
The investigation of linear and nonlinear dependence structures between the elements in an arbitrary family of random objects are inherent in many problems, ranging from modelling dependence within Markov processes (see e.g. \cite{Darsow1992},\cite{Lageraas2010},  \cite{Ibragimov2009}), Gaussian processes, and general processes with continuous marginals (see e.g. \cite{WilsonGharamani2010}), to the modelling of dependence between semimartingale processes (see e.g. \cite{Kallsen2006},\cite{Bielecki2010}) or the components of a random measure 
(see e.g. \cite{Blei2018}). One of the most powerful tools, which captures the whole structure of statistical dependence for a finite collection of real-valued random variables, are copulas. 

Copulas are cumulative distribution functions with uniform marginals, which can generally be interpreted as the dependence structure separated from the laws of the marginals by virtue of Sklar's Theorem.
The theory for copulas is 
rather well-developed for the 
finite-dimensional case (see \cite{Nelsen2006} for an introduction into the topic). In this paper we 
develop a general theory of copulas in infinite dimensions.

Relying on the one-to-one correspondence of probability measures and distribution functions in finite dimensions, we will introduce copulas as probability measures with uniform marginal distributions on product spaces. That is, copulas are treated as laws of families $
(U_i)_{i\in I}$ 
where $I$ is an arbitrary index set and $U_i\in L^0(\Omega;\mathbb{R})$ are real-valued uniformly distributed random variables on a probability space $(\Omega,\mathcal{F},\mathbb{P})$.

We prove Sklar's theorem in this general setting: The first part of this theorem states that each probability law on $\mathbb{R}^I$ possesses an underlying copula measure (representing its dependence structure), whereas the second part enables us to merge together any copula measure with a freely chosen family of  one-dimensional marginals 
to a law on the product space (with the copula measure as the specified dependence structure).
The framework here suggested is well suited for the general setup of real-valued stochastic processes.
Nevertheless, when setting an eye towards applications, e.g. numerical approximations or (functional) data analysis, it is relevant to have sufficient knowledge about  various properties like integrability or regularity of the processes. Thus it is inevitable to consider the construction of and the inference for measures in (topological) function spaces.

Unfortunately, the advent of copulas in function spaces is subject to some nontrivial difficulties compared to the finite-dimensional setting:
First of all, it is not immediately clear what marginals are in an infinite-dimensional vector space.
If $X$ is a random variable in a Hilbert space $H$, projections onto an orthonormal basis $(\langle X,e_n\rangle_H)_{n\in \mathbb N}$ are immediate candidates. Nevertheless, if, in addition, the space considered is a reproducing kernel Hilbert space of functions, say over $[0,1]$, an equally natural option for marginals are function evaluations $(X(t))_{t\in [0,1]}$. This motivates the introduction of a flexible framework, which is needed, but not yet existing. We therefore propose a general concept of marginals for measurable vector spaces.

Another critical point in the infinite-dimensional setting is that even if we fix a certain notion for marginals and then construct a measure with some given dependence structure (i.e. a copula) and marginals via Sklar's theorem, this measure is not necessarily a Borel measure on the desired function space, but may rather be just a 
cylindrical premeasure. Proving whether a cylindrical premeasure actually corresponds to a proper probability measure is a considerably difficult task. We will refer to this as the \textit{construction problem}, in the context of applying the second part of Sklar's Theorem. 
In applications, we further wish to be flexible in the choice of copulas and marginals and hence want to avoid to be overly restrictive by stating complicated conditions on the mutual behaviour of dependence structure and one-dimensional distributions.
 Otherwise useful criteria, as for instance those based on compactness arguments (e.g. Theorem 6.2 in \cite{Buldygin2000}), may turn to be be cumbersome to translate into feasible criteria to overcome the construction problem. Nevertheless, there are several important situations in which one can find a satisfying framework to work with. A major part of our work describes respective constructions, namely in the space of continuous functions $C(T)$, in Hölder spaces, in the Lebesgue spaces $L^p(T)$, and in the sequence spaces $l^p$. 

For the latter two cases simple moment criteria on the marginals prove to be sufficient (and sometimes even necessary), what makes them attractive in practice.
The suitability of these criteria is confirmed by the characterizing connections to Wasserstein spaces under these conditions, as copulas could equivalently be introduced as optimal solutions to a restricted optimization problem in the corresponding Wasserstein spaces. Even more, in the case that two random variables over $l^p$ or $L^p(T)$ respectively share the same copula, we can express the corresponding Wasserstein distance in terms of the Wasserstein distances of their one-dimensional marginal distributions (and hence in closed form).
This is apparently useful, when we want to approximate distributions for which we have the corresponding copula at hand, but also in the presumably more frequently encountered situation, 
in which both marginals and copula are unknown. 
In fact, we can conveniently bound the $L^p(T)$ distance of two random variables from above by the Wasserstein distance of their one-dimensional marginals and the $L^q(T)$-distance of the corresponding underlying copula processes (i.e. processes that have the corresponding copulas as their laws) for any $q\geq 1$,  
under suitable smoothness and tail assumptions on the marginals of one of the variables. We finally demonstrate how to apply this bound in order to approximate (heavy-tailed and tail-dependent) stochastic processes with underlying elliptical copula and regularly varying marginals.

The paper is organised as follows.
We describe the basic framework of copulas in product spaces and prove Sklar's Theorem in section \ref{sec: Product space copulas}. Section \ref{sec: Copulas in function spaces} is devoted to copula constructions in function spaces, where in subsection \ref{Marginals and the construction Problem} we introduce a general framework for marginals in measurable vector spaces and describe the abstract construction problem. Subsection \ref{sec: Solutions to the construction problem} presents criteria to overcome the latter in various function spaces.
Finally, section \ref{sec: Robustness of the copula construction} provides distance estimates for the copula construction, where in subsection \ref{sec: Copulas and Wasserstein spaces} we describe the connection of copulas to Wasserstein spaces and in
subsection \ref{sec: Lp estimate} we derive an estimate of the $L^p(T)$ distance of two processes in terms of the difference of the underlying copula and the one-dimensional Wasserstein distances of their marginals.

\subsection*{Notation} 
 For any measure $\mu$ on a measurable space $(B,\mathcal B)$ and a measurable function $f:(B,\mathcal B)\to (A,\mathcal A)$ into another measurable space $(A,\mathcal A)$ we denote by $f_*\mu$ the pushforward measure 
 with respect to $f$ given by $f_*\mu(S):=\mu( f^{-1}(S))$ for all $S\in \mathcal A$.
 If $B=\mathbb{R}^I$, where $I$ is an arbitrary index set, and $\mathcal B=\otimes_{i\in I} \mathcal{B}(\mathbb R)$, we use the shorter notations $\pi_{J*}\mu=:\mu_J$ for a subset $J\subseteq I$ and $\pi_{\lbrace i\rbrace*}\mu=:\mu_i$ for an element $i\in I$, where $\pi_J$ denotes the projection on $\mathbb{R}^J$.  If $J\subset I$ is finite, we denote the corresponding finite-dimensional cumulative distribution functions by $F_{\mu_J}$ or $F_{\mu_i}$ respectively.
 We will frequently refer to the one-dimensional distributions $\mu_i,i\in I$ and equivalently $F_{\mu_i},i\in I$ as marginals of the measure $\mu$. 
 Throughout the paper all random variables are considered on a complete probability space $(\Omega,\mathcal F,\mathbb P)$ and we write $L^0(\Omega,\mathcal F;A,\mathcal A)=:L^0(\Omega;A)$ for the measurable functions $f:(\Omega,\mathcal F)\to (A,\mathcal A)$, i.e., $A$-valued random variables.

\section{Copulas and Sklar's Theorem in Infinite Dimensions}\label{sec: Product space copulas}
 Following the natural interpretation of copulas as measures in finite dimensions (see section \ref{sec: Copulas in finite dimensions} for a short treatment of copulas in finite dimensions), we suggest to define the concept in the same line also in infinite dimensions:
\begin{Definition}\label{T: Consistent Copulas}
A 
 copula measure (or simply copula) on $\mathbb{R}^I$ is a probability measure $C$ on $\otimes_{i\in I} \mathcal{B}(\mathbb R)$, such that its marginals $C_i$ are uniformly distributed on $[0,1]$.
\end{Definition}
For finite-dimensional index sets $I$ the notions of measures and cumulative distribution functions have a one-to-one correspondence, which is the reason why in this case a copula measure $C$ can be uniquely identified with the copula $F_C$
 in the classical sense of Definition \ref{D: Finite dimensional Copula}. For the same reason the finite-dimensional distributions $C_J$ of an infinite-dimensional copula measure $C$, correspond uniquely to the copula $F_{C_J}$ in the familiar sense of copulas in finite dimensions.
 
We also introduce the important notion of \textit{copula processes}.
\begin{Definition}
We call a random variable $U\in L^0(\Omega;\mathbb R^I)$ 
with uniform marginals on $[0,1]$ a \textit{copula process}. That is, the law of a copula process is a 
 copula measure.
 \end{Definition}
 Since for each copula we can find a probability space, and a copula process with law $C$ on it, the notion of copulas has a one-to-one correspondence with the one of copula processes.

As in finite dimensions, the most important result for the use of copulas is
Sklar's Theorem: 
\begin{Theorem}[Sklar's Theorem]\label{T: Sklar in infinite dimensions}
Let $I$ be an index set and $\mu$ be a probability measure on $\otimes_{i\in I} \mathcal{B}(\mathbb R)$ with marginal one-dimensional distributions $\mu_i, i\in I$.
There exists a copula measure $C$, such that for each finite subset $J\subseteq I$, we have
\begin{equation}\label{Sklar Property}
F_{C_J}\left(\left(F_{\mu_{j}}\left(x_{j}\right)\right)_{j\in J}\right)=F_{\mu_J}\left(\left(x_{j}\right)_{j\in J}\right)
\end{equation} 
for all $(x_{j})_{j\in J}\in \mathbb{R}^{ J}$. Moreover, $C$ is unique if $F_{\mu_{i}}$ is continuous for each $i\in I$.
Vice versa, let $C$ be a copula measure on $\mathbb{R}^I$ and let $(\mu_i)_{i\in I}$ be a collection of (one-dimensional) Borel probability measures over $\mathbb R$. 
Then there exists a unique probability measure $\mu$ on $\otimes_{i\in I} \mathcal{B}(\mathbb R)$, such that 
\eqref{Sklar Property} holds. 
\end{Theorem}
In the following proof and the rest of the paper we often use for a one-dimensional Borel measure $\mu_i$ on $\mathbb R$ the notation $F_{\mu_i}^{[-1]}$ for the quantile functions
\begin{equation}\label{Inverse Transform}
    F_{\mu_i}^{[-1]}(u) := \inf \left\lbrace x\in(-\infty,\infty) : F_{\mu_i}(x)\geq u\right\rbrace,
\end{equation}
\begin{proof}
To prove the first part, let $(X_i)_{i\in I}$ be a random vector having $\mu$ as its law. Let $U$ be a standard uniformly distributed real-valued random variable on the same probability space, such that $U$ is independent of $(X_i)_{i\in I}$. 
For a one-dimensional distribution function we denote its left-limit by
$F_{\mu_i}(x-):=\lim_{y\uparrow x}F_{\mu_i}(y).$ 
Define the distributional transform process $(U_i)_{i\in I}$ by
$$U_i:=F_{\mu_i}(X_i-)+U\left(F_{\mu_i}(X_i)-F_{\mu_i}(X_i-)\right)$$
and $C$ to be the law of $(U_i)_{i\in I}$. Since each $U_i$ is uniformly distributed on $[0,1]$ and the finite-dimensional laws $C_J$ fulfill \eqref{Sklar Property} by Theorem \ref{T: Rüschendorf Transform}, $C$ is the copula measure we looked for.
Observe that in case of continuous marginals all finite-dimensional marginals of $C$ must be uniquely determined by the unique copulas of the finite-dimensional laws of $\mu$ induced by Sklar's Theorem in finite dimensions.

To prove the other direction of Sklar's Theorem,
observe that, since $F_{\mu_i}^{[-1]}:[0,1]\rightarrow\mathbb R$ is measurable for every $i\in I$ 
we have that $(F_{\mu_i}^{[-1]})_{i\in I}$ 
 is a measurable map from the product space $([0,1]^{I},\otimes_{i\in I}\mathcal{B}([0,1]),C)$ to $(\mathbb{R}^{I},\otimes_{i\in I}\mathcal B (\mathbb{R}))$.
 The measure $\mu$ on $\otimes_{i\in I}\mathcal{B}(\mathbb R)$ given by the corresponding pushforward measure
\begin{equation}\label{Explicit form of seond part Sklar}
\mu:=((F_{\mu_i}^{[-1]})_{i\in I})_*C 
\end{equation}
has the desired properties.
To see this, we just have to verify that $\mu$ has the finite-dimensional distributions induced by (\ref{Sklar Property}).
Observe that, for all $i\in I$, by the monotonicity of the cumulative distribution functions we have that, for all $x\in (-\infty,\infty)$, 
\begin{align*}
 \left\lbrace u\in [0,1]:F_{\mu_i}^{[-1]}(u)\leq x\right\rbrace \supseteq \left\lbrace u\in [0,1]:u< F_{\mu_i}(x)\right\rbrace = [0,F_{\mu_i}(x))
\end{align*}
 and
\begin{align*}
\left\lbrace u\in [0,1]:F_{\mu_i}^{[-1]}(u)\leq x\right\rbrace\subseteq \left\lbrace u\in [0,1]:u\leq F_{\mu_i}(x)\right\rbrace=[0,F_{\mu_i}(x)].
\end{align*}
Thus, for $J\subset I$ finite, we have for all $(x_j)_{j\in J}\in\mathbb{R}^J$ that
\begin{align*}
\left([0,F_{\mu_{j}}(x_j)]\right)_{j\in J}\setminus \left(\left\lbrace u\in [0,1]:F_{\mu_{j}}^{[-1]}(u)\leq  x_j\right\rbrace\right)_{j\in J}
\subseteq \left(\left\lbrace F_{\mu_{j}}(x_j)\right\rbrace\right)_{j\in J}
\end{align*}
is a $C_J$ nullset. Therefore we obtain
\begin{align*}
  C_J\left(\left(\left(F_{\mu_j}^{[-1]}\right)^{-1}(-\infty,x_1]\right)_{j\in J}\right)
   = & C_J\left(\left(\left\lbrace u\in [0,1]:F_{\mu_j}^{[-1]}(u)\leq x_j\right\rbrace\right)_{j\in J}\right)\\
   = & C_J\left(\left([0,F_{\mu_j}(x_j)]\right)_{j\in J}\right) \\=&F_{C_J}\left(F_{\mu_j}\left(\left(x_j\right)\right)_{j\in J}\right).
\end{align*}
This concludes the proof.
\end{proof}
\begin{Remark}
If $I$ is a finite set Theorem \ref{T: Sklar in infinite dimensions} coincides with Sklar's Theorem \ref{T: Finite Sklar} in finite dimensions  by identifying the copula measure uniquely with its corresponding cumulative distribution function.
\end{Remark}
\begin{Remark}
From the proof above it follows that for a copula measure $C$ on $\mathbb{R}^I$ and a collection of marginals $(\mu_i)_{i\in I}$,
the pushforward measure in  \eqref{Explicit form of seond part Sklar} represents a probability measure $\mu$ on $\mathbb{R}^I$ having this underlying copula and marginals.
\end{Remark}
In \cite{CuestaAlbertos1993} the authors used the notion that two laws $\mu$ and $\nu$ on $(\mathbb{R}^I,\otimes_{i\in I}\mathcal B (\mathbb R))$ have the same dependence structure, if there exist two stochastic processes $X=(X_i)_{i\in I}$ and $Y=(Y_i)_{i\in I}$, such that $X\sim \mu$ and $Y\sim \nu$ on the same probability space $(\Omega,\mathcal F,\mathbb P)$ and $X_i$ and $Y_i$ are similarly ordered ($X_i\overset{\text{s.o.}}{\sim}Y_i$) for all $i\in I$, that is $$\left(X_i(\omega)-X_i(\omega')\right)\left(Y_i(\omega)-Y_i(\omega')\right)\geq 0 \quad  \mathbb P \otimes \mathbb P \text{ a.s.}.$$ 
In finite dimensions this notion is equivalent to the existence of a common underlying copula by virtue of Sklar's Theorem \ref{T: Finite Sklar}. This is also valid in infinite dimensions, as we show next.
Later, this fact will play a crucial in transferring the theory for optimal couplings of stochastic processes as treated in \cite{CuestaAlbertos1993} to our copula setting and in proving therewith approximation results in section \ref{sec: Robustness of the copula construction}.
 \begin{Lemma}\label{Lem: Same copulas means similarly ordered}
 Two probability measures $\mu$ and $\nu$ on $\otimes_{i\in I}\mathcal B(\mathbb R)$ have a common underlying copula measure in the sense of (\ref{Sklar Property}) if and only if $X_i\overset{\text{s.o.}}{\sim}Y_i$ for all $i\in I$.
 \end{Lemma}
\begin{proof}
Let $\mu$ and $\nu$ have the same underlying copula measure $C$ and let $U\sim C$ be a corresponding copula process. Define with the notion introduced in \eqref{Inverse Transform} the random variables $X:=(F_{\mu_i}^{[-1]}(U_i))_{i\in I}$ and $Y=(F_{\nu_i}^{[-1]}(U_i))_{i\in I}$. By construction and analogously to the proof of Sklar's Theorem \ref{T: Sklar in infinite dimensions}, we obtain $X\sim ((F_{\mu_i}^{[-1]})_{i\in I})_*C=\mu$ and $Y\sim ((F_{\nu_i}^{[-1]})_{i\in I})_*C=\nu$. Since the quantile transforms $F_{\mu_i}
^{[-1]}$ and $F_{\nu_i}
^{[-1]}$ are nondecreasing functions, we obtain that $X_i\overset{\text{s.o.}}{\sim}Y_i$ for all $i\in I$.

Vice versa, let $X\sim \mu$ and $Y\sim \nu$ be two random variables such that $X_i\overset{\text{s.o.}}{\sim}Y_i$ for all $i\in I$. Then by Proposition 2.1 in \cite{CuestaAlbertos1993} for each $i\in I$ there exist a uniformly distributed random variable $U_i$ such that $X_i=F_{\mu_i}
^{[-1]}(U_i)$ and $Y_i=F_{\nu_i}
^{[-1]}(U_i)$ (observe that the proof of this assertion does not need second moments, as stated in Remark 1 in \cite{CuestaAlbertos1993}). If $C$ is the law of $U=(U_i)_{i\in I}$, we obtain $\mu=((F_{\mu_i}
^{[-1]})_{i\in I})_*C$ and $\nu=((F_{\nu_i}
^{[-1]})_{i\in I})_*C$. This shows, that $X$ and $Y$ have the same underlying copula measure $C$.
\end{proof}
The following examples review some existing concepts of copulas, which can be embedded into our framework: 
\begin{Example}\label{Ex: Complete and Independence Copula}(Complete dependence and independence copulas) The complete dependence copula measure on $\mathbb{R}^I$ is the law corresponding to the consistent family of finite-dimensional cumulative distribution functions given by $M_J((u_j)_{j\in J})=\min_{j\in J} u_j$. Observe, that its finite-dimensional distribution functions are Fr{\'e}chet-Hoeffding upper bounds for the corresponding finite-dimensional copulas, that is, for all $J\subset I$ finite and any copula $C$ on $\mathbb{R}
^I$ we have $$F_{C_J}\left(\left(u_j\right)_{j\in J}\right)\leq M_J\left(\left(u_j\right)_{j\in J}\right) \quad \forall (u_j)_{j\in J}\in [0,1]^J.$$

  The independence copula measure on $\mathbb{R}^I$ is the law of the consistent family of finite-dimensional cumulative distribution functions given by $\Pi_J((u_j)_{j\in J})=\Pi_{j\in J} u_j$.
\end{Example}

\begin{Example}
(Inversion method and Gaussian copulas)
Given a law $\mu$ with continuous marginals $F_{\mu_i}, i\in I$, the underlying copula measure $C$ induced by Sklar's Theorem \ref{T: Sklar in infinite dimensions} is given by its finite-dimensional cumulative distribution functions for each finite $J\subseteq I$ via
\begin{equation}
    F_{C_J}\left(\left(u_j\right)_{j\in J}\right):=F_{\mu_J}\left(\left(F_{\mu_j}^{-1}\left(u_j\right)\right)_{j\in J}\right)\quad \forall (u_j)_{j\in J}\in [0,1]^J.
\end{equation}
This method is known as inversion method (see e.g. \cite{Nelsen2006}).
In this way we can derive, for instance, the copula measures that are underlying a Gaussian process (that is, each $\mu_J$ is Gaussian), which are called Gaussian copulas. 
Infinite-dimensional Gaussian copulas where applied already for example in \cite{WilsonGharamani2010} in a machine learning context.
\end{Example}

\begin{Example}(Archimedean  copulas):   Fix a continuous, strictly decreasing and convex function $\phi:[0,1]\to[0,\infty]$ such that $\phi(1)=0$, and $\phi^{[-1]}$, its pseudoinverse, is given by
$$\phi^{[-1]}(x):=\begin{cases}
\phi^{-1}(x) & 0\leq x\leq \phi(0)\\
0 & x>\phi(0).
\end{cases}
$$
Then the finite-dimensional laws of an Archimedean copula measure are given by
$$F_{C_J}\left(\left(u_j\right)_{j\in J}\right):=\phi^{[-1]}\left(\sum_{j\in J} \phi\left(u_j\right)\right)$$
for each finite $J\subset I$.
By definition, these probability measures are exchangeable and it was shown in \cite{Constantinescu2011} that Archimedean copulas in infinite dimensions can be related to Dirichlet distributions.
\end{Example}

In addition, our framework accommodates also Markov copulas, introduced in \cite{Darsow1992} and developed also, e.g., in \cite{Lageraas2010}, \cite{Ibragimov2017}, \cite{Ibragimov2009}, and \cite{Bibbona2016}, copulas for time series introduced in \cite{Cherubini2012} and copulas in Hilbert spaces from \cite{MR3574701}.

\section{Copulas in Function spaces}\label{sec: Copulas in function spaces}
In this section we formulate a unified setting for the notion of copulas in the framework of vector spaces. 
 
 \subsection{Marginals in Vector Spaces}\label{Marginals and the construction Problem}
Let $V$ be a vector space over $\mathbb R$ 
and $\mathcal V$ a $\sigma$-algebra over this space. Recall that the algebraic dual of $V$ is defined as the vector space
\begin{equation*}
    Hom(V,\mathbb R):=\left\lbrace \varphi:V\to \mathbb R: \varphi\text{ is linear}\right\rbrace.
\end{equation*} 
\begin{Definition}[$M$-Marginals]\label{D: W-Marginals}
Let $X$ be a random variable on $V$. Let $M$ be a linearly independent subspace of measurable functions in $Hom(V,\mathbb R)$ that separates the points $V$. Then we call the random variables $(m(X):m\in M)$ the $M$-marginals of $X$.
\end{Definition}
Observe that by the definition above, we are able to embed the vector space framework into the framework of product spaces, by the embedding
\begin{equation}\label{Product space embedding}
    V\ni v\mapsto \left(m\left(v\right)\right)_{m\in M}\in \mathbb{R}^{M},
\end{equation}
which is necessary for the application of our copula theory. 

Some choices of $M$ which are of practical importance are given in the sequel. 
\begin{Example}[Marginals in finite dimensions]
In the finite-dimensional case, that is $V=\mathbb{R}^d$ for some $d\in\mathbb N$, $M$ is necessarily of the form
\begin{equation}\label{W Marginals for finite dimensions}
    M=\left\lbrace \langle e_1,\cdot\rangle,...,\langle e_d,\cdot\rangle\right\rbrace.
\end{equation}
for a basis $e_1,...,e_d$ of $\mathbb{R}^d$ and where $\langle \cdot,\cdot\rangle$ denotes an inner product on $\mathbb{R}^d$.
In terms of finite-dimensional copula theory, the natural choice is the standard basis $e_1=(1,0....,0)$, ... , $e_d=(0,...,0,1)$.
\end{Example}
\begin{Example}[Product space]\label{Ex: Marginals for stochastic processes}
It is possible to embed the product space setting from section \ref{sec: Product space copulas} into the framework of measurable vector spaces:
The product space $V=\mathbb{R}^I$ for some index set $I$ becomes a measurable vector space, if we equip it with the product $\sigma$-algebra $\otimes_{i\in I}\mathcal{B}(\mathbb R)$. The projections (or evaluation functionals) $\pi_j((v_i)_{i\in I}):=v_j$ for $j\in I$ are measurable (even continuous) by definition, linearly independent and separate the points. Thus, we can take
\begin{equation}\label{W Marginals for product spaces}
    M=\left\lbrace \pi_i:i\in I\right\rbrace.
\end{equation}
Observe that we can do this with every space of functions, in which the evaluations are linearly independent. For instance we can take the space of $p$-integrable functions over a subset $T$ of $\mathbb R^d$, $d\in\mathbb N$
\begin{align}
  \mathcal{L}^p(T) :=& \mathcal{L}^p(T,\mathcal A, \mu;\mathbb R)\notag\\
    =& \left\lbrace f:T\to \mathbb R: f\text{ is measurable and }\|f\|_{\mathcal{L}^p(T)}:=(\int_T f(t)^p \mu(dt))^{\frac 1p}<\infty \right\rbrace
\end{align}
 for some natural number $p$ and a measure space $(T,\mathcal A,\mu)$.
 Observe that in this setting we work in a space of functions, rather than of equivalence classes. The reason is that point evaluations are not well defined in Banach spaces of equivalence classes. This serves also as motivation for 
subsection \ref{sec: pretty Lp case} where we construct copulas under these circumstances.
\end{Example}
\begin{Example}[Path marginals for Banach spaces of functions]\label{Ex: C marginals} Let $V$ be a separable Banach space of real-valued functions on a set $T$ such that the evaluation functionals $\delta_t f:=f(t)$ (or projections in terms of product spaces) are continuous and $\mathcal V =\mathcal B (V)$ is the Borel $\sigma$-algebra with respect to the corresponding norm topology.
In most of these settings, the subset 
\begin{equation}
M=\left\lbrace \delta_t:t\in T\right\rbrace
\end{equation}
of evaluations is linearly independent and, due to continuity, it consists of measurable functionals.
Important examples in this framework are the continuous functions $V=C(T)$ and $V=B_K$, where $B_K$ is a reproducing kernel Banach or Hilbert space in the sense of  \cite{Zhang2009} or \cite{Berlinet2011}. 
Observe that, if we revisit Example \ref{Ex: Marginals for stochastic processes} allowing for general topological vector spaces $V$ instead of Banach spaces, then Example \ref{Ex: Marginals for stochastic processes} would be part of this framework.
\end{Example}
\begin{Example}(Basis marginals)\label{Ex: Basis marginals}
If $V$ is a Banach space that possesses a Schauder basis (cf. Definition \ref{Def: Schauder Basis}) we can take 
 \begin{equation}
    M=\left\lbrace f_n:n\in \mathbb N\right\rbrace.
\end{equation}
where $(f_n)_{n\in \mathbb N}$ is the sequence of coefficient functionals of the Schauder basis. 
Examples of Banach spaces that possess such a basis are $C([0,1])$, $L^p([0,1])$, the sequence spaces $l^p$ and, as a special case, all separable Hilbert spaces with orthonormal bases as Schauder bases. Note that in the latter case we are effectively in the setting of \textit{consistent copulas} from \cite{MR3574701}.
\end{Example}
\begin{Remark}[Marginals for nonlinear subspaces]\label{Ex: random cdf marginals}
Note, that if we are just interested in defining random variables on particular subsets of a vector space $V$, the set $M$ must not necessarily be separating for all elements of $V$. 
One example is the construction of random probability measures, as a certain subset of random variables in the Banach space of signed measures on the real line. 
In this case it suffices to take 
\begin{equation}\label{CDF Marginals}
    M=\left\lbrace F_{\cdot}(t): F_{\mu}(t)= \mu(-\infty,t], t\in\mathbb R\right\rbrace,
\end{equation}
that is, we identify a random probability measure with the corresponding random cumulative distribution function.
\end{Remark}
We will refer to the choice of marginals in Examples \ref{Ex: Marginals for stochastic processes} and \ref{Ex: C marginals} as \textit{path marginals} and the corresponding copulas in this framework as \textit{path copulas}. In contrast, the corresponding constructions in Example \ref{Ex: Basis marginals} will be referred to as \textit{basis marginals} and \textit{basis copulas}. 

Already in finite dimensions, due to different basis specifications, there is not just one choice for $M$. Unfortunately, copulas are not invariant under change of the notion of marginals, as shown by the following example:
\begin{Example}
Suppose $(X_1,X_2)$ and $(Y_1,Y_2)$ are two bivariate real random variables on the same probability space, 
such $C_{X_1,X_2}(u,v)=C_{Y_1,Y_2}(u,v)=uv$ is the independence copula. Assume, moreover, that $X_1\sim N(0,1)$ and $X_2,Y_1,Y_2\sim U(0,1)$. 
By Proposition 3.4.1 in \cite{Cherubini2012} we know that
$$C_{X_1,X_1+X_2}(u,v)=\int_0
^u \frac{d}{d_{x_1}}C_{X_1,X_2}\left(w,F_{X_2}\left(F_{X_1+X_2}^{[-1]}(v)-F_{X_2}^{[-1]}(w)\right)\right)dw.$$
Since the independence copula is simply the product of the one-dimensional uniform distributions and the distribution function of $X_2$ is the identity on $[0,1]$, we have
\begin{align*}
C_{X_1,X_1+X_2}(u,v)= & \int_0
^u F_{X_2}\left(F_{X_1+X_2}^{[-1]}(v)-F_{X_2}^{[-1]}(w)\right)dw \\
= &\int_0
^u \left(F_{X_1+X_2}^{[-1]}(v)-w\right)dw\\
= & u F_{X_1+X_2}^{[-1]}(v)-\frac{u^2}{2}
\end{align*}
and analogously $C_{Y_1+Y_2,Y_2}(u,v)=u F_{Y_1+Y_2}^{[-1]}(v)-\frac{u^2}{2}$. This induces that $C_{Y_1+Y_2,Y_2}$ and $C_{X_1+X_2,X_2}$ coincide if and only if 
$$F_{X_1+X_2}^{[-1]}(v)=F_{Y_1+Y_2}^{[-1]}(v)\quad \forall v\in \mathbb R,$$
which is obviously not the case, in view of the distributional choices on the random variables.
Thus, $(X_1,X_2)$ and $(Y_1,Y_2)$ do not posses the same copula with respect to $\lbrace (1,1),(0,1)\rbrace$-marginals, although they share the same copula with respect to $\lbrace (1,0),(0,1)\rbrace$-marginals.
\end{Example}
If we want to construct a measure on a vector space $V$ by virtue of the second part of Sklar's Theorem the naive procedure reads now as follows:
 \begin{Construction}\label{abstract construction}\hspace{10cm}
\begin{itemize} 
    \item[(i)] Choose some set $M$ which satisfies the conditions of Definition \ref{D: W-Marginals}.
    \item[(ii)] Choose a copula $C$ on $\mathbb{R}^M$ (or a copula process $(U_m)_{m\in M}$) and one-dimensional distributions $(\mu_m)_{m\in M}$ and merge them with Sklar's Theorem to a law $\mu$ (or a process) on $\otimes_{m\in M
}\mathcal B (\mathbb R)$. 
\item[(iii)] (Construction Problem) Check if $\mu$ can be identified with a measure on $\mathcal V$ via the embedding \eqref{Product space embedding}.
\end{itemize}
\end{Construction}
As anticipated in the introduction, the third point will not necessarily carry an affirmative answer. The choice of marginals and dependence structure in (ii) must be based on criteria that guarantee a solution to (iii), which is hereafter referred to as the \textit{construction problem} for copulas in function spaces.

Consider now the following framework (which covers all mentioned examples): $V$ is a topological vector space, $\mathcal V=\mathcal B(V)$ the corresponding Borel $\sigma$-algebra and $M$ a subset of the dual that satisfies the conditions of Definition \ref{D: W-Marginals}. In addition,  assume that each $m\in M$ is continuous, that is, $$M\subset V^*,$$ where $V^*$ denotes the topological dual of $V$, given by
$$V^*:=\left\lbrace v^*:V\to\mathbb{R}: v^*\text{ is linear and continuous}\right\rbrace.$$
Then Construction \ref{abstract construction} induced by Sklar's Theorem effectively culminates in the construction of a cylindrical premeasure on that vector space (see for instance \cite{Buldygin2000} or \cite{Schwartz1973} for a treatment of cylindrical measure theory).
In the case that $V$ is even a separable Banach space and $M\subset V^*$, we however have the following useful criterion for our setting:
\begin{Lemma}\label{L: Measurability is not a problem in separable Banach spaces}
Let $V$ be a separable Banach space. 
Assume that $M\subset V^*$ is a fundamental set with respect to the $\text{weak}^*$-topology, that is, its linear span is dense.
If the probability measure defined in Construction \ref{abstract construction} is the law of a process $X:=(X_m)_{m\in M}$, such that $X$ is almost surely in the range of the embedding \eqref{Product space embedding}, then it is the image of a Borel measurable random variable $\tilde X$ in $V$ under this embedding. 
\end{Lemma}
\begin{proof}
If $(X_m)_{m\in M}$ is almost surely in the range of the embedding, there is an $\tilde{\Omega}\subseteq \Omega$ with full measure and a random variable $\tilde X$ such that $m(\tilde X(\omega))=X_m(\omega)$ for all $\omega\in \tilde{\Omega}$. 
Since $M$ is a fundamental set, we have that for all $v^*\in V^*$ there is a sequence $(\sum_{i=1}^{N_n} \lambda_i^n m_i^n)_{n\in \mathbb N}$ in $lin(M)$ such that $\sum_{i=1}^{N_n} \lambda_i^n m_i^n\to v^*$ with respect to the $\text{weak}^*$-topology. Thus 
\begin{equation*}
    v^*(\tilde X)=\lim_{n\to\infty} \sum_{i=1}^{N_n} \lambda_i^n m_i^n(\tilde X) \qquad \text{a.s.}
\end{equation*}
is measurable, since linear combinations and limits of measurable functions are measurable. We conclude that $\tilde X$ is a weakly measurable random variable on a separable Banach space and hence, by the Pettis theorem \cite[Theorem 1.1]{MR1501970} strongly measurable, that is, measurable with respect to the Borel $\sigma$-algebra. 

\end{proof}
\begin{Remark}
Due to the existence of Hamel bases on $V^*$ and the Hahn-Banach Theorem (cf. Corollary 5.80 in \cite{MR2378491}), the existence of a set $M$ that satisfy the conditions in Definition \ref{D: W-Marginals}  
is always guaranteed in locally convex Hausdorff spaces. 
\end{Remark}
We will concentrate in the next sections 
on special cases of path- and basis constructions, which are adequate to solve the construction problem (for instance by virtue of Lemma \ref{L: Measurability is not a problem in separable Banach spaces}), which is why
they are foremost of practical importance.

\subsection{Solutions to the Construction Problem}\label{sec: Solutions to the construction problem}
\subsubsection{Path Copulas for $p$-Integrable Stochastic Processes}\label{sec: pretty Lp case}
We describe in this section how the copula construction induced by Sklar's Theorem \ref{T: Sklar in infinite dimensions} works for the function space $\mathcal{L}^p(T, \mathcal B(T), \mu;\mathbb R)=:\mathcal{L}^p(T)$ for $p\in\mathbb N$, a measurable set $T\subset \mathbb{R}^d$ with $d\in \mathbb N$ and a $\sigma$-finite measure $\mu$.
As mentioned in Example \ref{W Marginals for product spaces}, we take $M=\lbrace \delta_t:t\in T\rbrace$-marginals, that is, we identify a function $f\in\mathcal L^p(T)$ by all its function values $(f(t))_{t\in T}$.
Moreover, we denote by $[f]$ the corresponding equivalence class of almost everywhere coinciding functions with $f$, which forms an element in the Banach space of equivalence classes $L^p(T)$.

For a stochastic process $X=(X_t)_{t\in T}$ we say that it is measurable, if the mapping $(t,\omega)\mapsto X_t(\omega)$ is $\mathcal B (T)\otimes \mathcal A/ \mathcal B (\mathbb R)$-measurable.
\begin{Lemma}\label{L:Processes that are p integrable}
Let $X=(X_t)_{t\in T}$ be a measurable stochastic process.
\begin{itemize}
    \item[(a)] Assume $X$ has values in $\mathcal{L}^p(T)$ almost surely. Then $X$ is a Borel measurable random variable in $\mathcal{L}^p(T)$ (with respect to the pseudometric induced by $\|\cdot\|_{\mathcal L^p(T)}$).
\item[(b)] Let $X$ be measurable. Then
$[X]\in L^p(\Omega;L^p(T))$ if and only if 
\begin{equation*}
\int_T \mathbb{E}\left[|X_t|^p\right] \mu(dt)<\infty.
\end{equation*}
\end{itemize}
\end{Lemma}
\begin{proof}
We will verify that $[X]$ is a Borel measurable random variable on the Banach space $L^p(T)$. In that case we have that since $\mathcal{O}\subset \mathcal B(\mathcal{L}^p(T))$ is open if and only if $[\mathcal{O}]\subset \mathcal B(L^p(T))$ is open, thus $X$ is $\mathcal A/\mathcal{B}(\mathcal{L}^p(T))$-measurable if and only if $[X]$ is $\mathcal A/\mathcal{B}(L^p(T))$-measurable.
By Pettis theorem \cite[Theorem 1.1]{MR1501970} we have that $[X]$ is measurable, if and only if $\int_T X(t) y(t)dt$ is measurable for all $y\in L^q(T)$ with $q= \frac{p}{p-1}$ if $p\geq 2$ and for all $y\in L^{\infty}(T)$ if $p= 1$. Due to measurability of the process $X$, these integrals are indeed measurable. This shows (a).

To show (b), observe that by Fubini's theorem we have
$$\mathbb E \left[\int_T |X_t|^p \mu(dt)\right]=\int_T \mathbb{E}\left[|X_t|^p\right] \mu(dt)$$
whenever one of the terms in this equation is finite. Using (a), this shows the assertion.
\end{proof}

Lemma \ref{L:Processes that are p integrable} yields the following simple construction of random variables $X$ such that $[X]\in L^p(\Omega;L^p(T))$:
\begin{Construction}\label{Construction for Lp}\hspace{10cm}
\begin{itemize}
    \item[(i)] Specify a measurable copula process $U=(U_t)_{t\in T}$.
    \item[(ii)] define marginals $(F_t)_{t\in T}$, with corresponding $p$th moments $(m_t^p)_{t\in T}$, such that $(t,x)\mapsto F_t(x)$ is jointly measurable and
    \begin{equation}\label{pretty lp moment condition}
        \int_Tm_t^p\mu(dt)<\infty.
    \end{equation} 
    \item[(iii)] construct the new process $X$ with underlying copula process $U$ and marginals $(F_t)_{t\in T}$ via Sklar's Theorem \ref{T: Sklar in infinite dimensions}, that is 
    $$X_t=F_t^{[-1]}(U_t)\quad \forall t\in T.$$
    This process has values in $\mathcal{L}^p(T)$ and by Lemma \eqref{L:Processes that are p integrable}, $[X]$ is therefore an element in $L^p(\Omega;L^p(T))$.
\end{itemize}
\end{Construction}
Notice that the interpretability of the underlying path copula of $X$ is complicated, if transfer to the equivalence class $[X]$. Indeed path copulas 
specify dependence between point evaluations of the random function, which are not well defined anymore for equivalence classes. 
If one really wants to specify dependence between equivalence classes, one should approach this by considering the notion of basis marginals, as described in Subsection \ref{Sec:Schauder Basis Case}.

From a measure theoretical point of view, Banach spaces in which evaluation functionals are well defined and continuous are favourable and we will discuss this in the sense of spaces of continuous functions in the next subsection.

\subsubsection{Path Copulas for Continuous Processes}
For a given interval $T\subset \R^d, d\in \mathbb N$, that is $T=I_1\times...\times I_d$ for some (eventually unbounded) one-dimensional intervals $I_1,...,I_d$, we want to establish a `Sklar-like' theorem in the space of real continuous functions $C(T):=C(T;\mathbb R)$. If $T$ is compact, we equip this with the norm $\|f\|_{\infty}:=\sup_{t\in T}|f(t)|$, making $C(T)$ a separable Banach space.

Recall that a process $X=(X_t)_{t\in T}$ with marginals $(F_t)_{t\in [0,1]}$ is continuous in distribution if, for all $t\in T$,
\begin{equation}\label{General continuity in distribution}
   \lim_{s\to t} F_s(x_0)=  F_t(x_0) \quad \text{for all continuity points }x_0\text{ of }F_t.
\end{equation}
If we assume that all the marginals $F_t,t\in T$ are continuous (in $x$), \eqref{General continuity in distribution} simplifies to the condition that 
\begin{equation}\label{Continuous marginal continuity in distribution}
(t,x)\mapsto F_t(x)\quad \text{is continuous in both variables separately}
\end{equation}
In the latter case we have even joint continuity:
\begin{Lemma}\label{L: Partially Monotonicity implies joint continuity}
Assume that the marginals $F_t,t\in T$ of an almost surely continuous process $X=(X_t)_{t\in T}$ are continuous. Then 
\begin{equation}\label{Continuous marginal continuity in distribution}
(t,x)\mapsto F_t(x)\quad \text{is jointly continuous}.
\end{equation}
If the marginals are strictly increasing between the points $F^{[-1]}_t(0+)$ and $F^{[-1]}_t(1)$ in $x$ we have that
\begin{equation}\label{Continuous quantile marginal continuity in distribution}
(t,x)\mapsto F_t^{[-1]}(x)\quad \text{is jointly continuous}.
\end{equation}
\end{Lemma}
\begin{proof}
Due to Lemma 21.2 from \cite{Vaart1998} we have that $t\mapsto F_t^{[-1]}$ is pointwise continuous.
The proof follows then analogously to the arguments of the proof of Proposition 1 in \cite{Kruse1969}.
\end{proof}
 Since processes with continuous sample paths are continuous in distribution, \eqref{General continuity in distribution} (resp. \eqref{Continuous marginal continuity in distribution}) forms a necessary condition on the marginals. 
 \begin{Theorem}\label{T: Sklar for continuous functions}
 Let $X=(X_t)_{t\in T}$ be a stochastic processes 
 with sample paths that belong almost surely to $C(T)$ and such that it has continuous marginals $F_t$ for all $t\in T$. Then  $U=(U_t)_{t\in T}$ defined by
 \begin{equation}\label{Continuous function copula process}
     t\mapsto U_t:=F_t(X_t)
 \end{equation}  
 is a copula process for $X$ and almost surely continuous on $T$.
Vice versa, if $U$ is a copula process that is almost surely continuous on $T$ and $F_t,t\in T$ are strictly increasing marginals between the points $F^{[-1]}_t(0+):=\lim_{x\downarrow 0}F^{[-1]}_t(x)$ and $F^{[-1]}_t(1)$, which are continuous in distribution, then  $Y=(Y_t)_{t\in T}$ defined by
\begin{equation}\label{Continuous Sklar construction}
   Y_t= F_t^{[-1]}(U_t)
\end{equation}
  is a random variable, which is almost surely in $C(T)$ with marginals $F_t$ and underlying copula process $U$. Moreover, if $T$ is compact, $Y$ is measurable with respect to the Borel $\sigma$-algebra on $C(T)$. 
 \end{Theorem}
 \begin{proof}
 Observe, that since we are in the case of continuous marginals, the process $(U_t)_{t\in T}$ defined by $U_t=F_t(X_t)$ for all $t\in T$ is a copula process underlying $X$. Its continuity follows by Lemma \ref{L: Partially Monotonicity implies joint continuity} and the continuity of $s\mapsto X_s$. 

 To show the second part, observe that $(t,x)\mapsto F^{[-1]}_t(x)$ is continuous in $x$, since the marginals $F_t,t\in T$ are strictly increasing between $F^{[-1]}_t(0+)$ and $F^{[-1]}_t(1)$. Hence $Y$ is a random variable with values in $C(T)$ almost surely by Lemma \ref{L: Partially Monotonicity implies joint continuity} and the continuity of $s\mapsto U_s$. Its Borel measurability for compact $T$ follows from Lemma \ref{L: Measurability is not a problem in separable Banach spaces}.
 \end{proof}
 \begin{Remark}
 In principle, a more abstract set $T$ could be taken, as the precise structure of $\mathbb R^d$ is not used, but for convenience we stay in the euclidean setting throughout this paper.
 \end{Remark}
\begin{Remark}
In the framework of stochastic processes, the initial value $X_0$ is often chosen to be deterministic. Therefore it has neither a continuous nor strictly increasing distribution in the initial value. Possibly, for some processes, we still manage to define a continuous underlying copula (if the copula process has a limit from above in $0$ which is uniformly distributed), but since this might be hard to check in general, it is reasonable to start the process a little bit later than in the origin. 
 \end{Remark}
\begin{Example}
For some $t_0>0$ let $X_{t}=B_{t}$ for $t\in [t_0,\infty)$ be a standard Brownian motion with sample paths in $C(\mathbb R_+)$.
$$U_t= F_{t}(B(t))= \frac{1}{\sqrt{2\pi t}}\int_{-\infty}^{B(t)} e^{-\frac{x^2}{2t}}dx.
$$
We note that the copula of a Brownian motion was investigated for instance in \cite{Sempi2016} in the framework of Markov copulas.
\end{Example}

\subsubsection{Path Regularity and Copulas}
Let again $T=I_1\times...\times I_d$ be an interval in $\mathbb R^d$ for some $d\in \mathbb N$.
Recall that for a constant $\gamma> 0$ a function $f:T\to \mathbb{R}$ is called locally $\gamma$-H{\"o}lder continuous, if for each $t\in T$ there is a neighbourhood $N(t)$ of $t$ in $T$ and a constant $K_t>0$, such that for all $s,r\in N(t)$ we have 
$$|f(s)-f(r)|\leq K_t |s-r|^{\gamma}.$$
For a nonnegative integer $k$, $\gamma \in (0,1]$ and $m\in\mathbb N$, we introduce the Hölder spaces $ C^{k,\gamma}(T;\mathbb{R}^m)$ to be the space of functions $f:T\to \mathbb{R}^m$ which are continuously differentiable up to order $k$ and the $k$th derivative is locally $\gamma$-Hölder continuous. 

Recall the following fact about locally Hölder continuous functions:
\begin{Lemma}\label{L: Hölder compositions}
Let $I_1,...,I_m$ be intervals and $f=(f_1,...,f_m)\in  C^{k,\gamma}(T;\mathbb{R}^m)$ and $g\in C^{l,\eta}(I_1\times...\times I_m;\mathbb{R})$ such that $f(T)\subset I_1\times...\times I_m$. Then 
$$g\circ f\in\begin{cases} C^{0,\gamma\eta}(T;\mathbb{R}) & k=l=0\\
C^{k,\gamma}(T;\mathbb{R}) & k>l\\
C^{l,\eta}(T;\mathbb{R}) & l>k\\
C^{k,\min(\gamma,\eta)}(T;\mathbb{R}) & l=k\geq 1.
\end{cases}$$
\end{Lemma}
\begin{proof}
This is a special case of Theorem 4.3 in \cite{Llave1998}.
\end{proof}
As a consequence of Lemma \ref{L: Hölder compositions} we obtain the following immediately:
\begin{Corollary}\label{Regularity Lemma for copulas I}
Let $X\in C^{k,\gamma}(T;\mathbb R)$ almost surely such that $(t,x)\mapsto F_{X_t}(x)\in C^{l,\delta}(T\times \mathbb R;\mathbb R)$. Let $U$ denote the associated copula process given by
$U_t=F_t(X_t), t\in T$. Then almost surely
\begin{equation*}
U\in \begin{cases} C^{0,\gamma\eta}(T;\mathbb R) & k=l=0\\
C^{k,\gamma}(T;\mathbb R) & k>l\\
C^{l,\eta}(T;\mathbb R) & l>k\\
C^{k,\min(\gamma,\eta)}(T;\mathbb R) & l=k\geq 1.
\end{cases}
\end{equation*}
For a copula process $U\in C^{k,\gamma}(T;\mathbb R)$ almost surely and marginal cumulative distribution functions $(G_t)_{t\in T}$,
such that $(t,u)\mapsto G_t^{[-1]}(u)\in C^{l,\delta}(T\times (0,1);\mathbb R)$, $Y$ denotes the process given by $Y_t=G_t^{[-1]}(U_t)$. Then we have almost surely that
\begin{equation*}
Y\in \begin{cases} C^{0,\gamma\eta}(T;\mathbb R) & k=l=0\\
C^{k,\gamma}(T;\mathbb R) & k>l\\
C^{l,\eta}(T;\mathbb R) & l>k\\
C^{k,\min(\gamma,\eta)}(T;\mathbb R) & l=k\geq 1.
\end{cases}
\end{equation*}
\end{Corollary}
 By virtue of the previous Corollary \ref{Regularity Lemma for copulas I} we can determine the regularity of a copula process underlying a fractional Brownian motion:
 \begin{Example}\label{L: Regularity of the Fractional Quantile Transform}
Assume that $(U_t)_{t\in (t_0,\infty)}$ is a copula process underlying a fractional Brownian motion $(B_{t}^H)_{t\in [t_0,\infty)}$ for some $t_0>0$ with Hurst parameter $H\in (0,1)$, that is, a centered Gaussian process with covariance function $$\mathbb{E}[B_t^H B_s^H]=\frac 12 \left(t^{2H}+s^{2H}+| t-s|^{2H}\right).$$
 The process $(U_t)_{t\in(t_0,\infty)}$ has locally $H$-H{\"o}lder continuous paths.
To see this, we just have to verify the local $H$-H{\"o}lder continuity of $(t,x)\mapsto\Phi^H_t(x)$ as stated in the Corollary \ref{Regularity Lemma for copulas I}, where we denoted by $\Phi^H_t$ the cumulative distribution functions of $B_t^H$.
We can estimate for $s,t\in [t_0,\infty)$ (with the constant $c=1/\sqrt{2\pi}$)
\begin{align*}
|\Phi^H_t(y)-\Phi^H_s(y)|=|\Phi_1^H\left(\frac{y}{t^H}\right)-\Phi_1^H\left(\frac{y}{s^H}\right)| 
=&|\int_{\min\left(\frac{y}{t^H},\frac{y}{s^H}\right)}^{\max\left(\frac{y}{t^H},\frac{y}{s^H}\right)} \frac{e^{-\frac{z^2}{2}}}{\sqrt{2 \pi}}dz\vert \\ \leq &  c|y| |t^{-H}-s^{-H}|\\
\leq & |y| \frac{c}{\min(t^{H},s^{H})^2} |t-s|^{H}\\
\leq & \frac{c|y|}{t_0^{2H}} |t-s|^{H}.
\end{align*}
Analogously, for $x,y\in\mathbb R$ we get 
\begin{align*}
|\Phi^H_t(x)-\Phi^H_t(y)|=|\Phi_1^H\left(\frac{x}{t^H}\right)-\Phi_1^H\left(\frac{y}{t^H}\right)|=&|\int_{\min\left(\frac{y}{t^H},\frac{x}{t^H}\right)}^{\max\left(\frac{y}{t^H},\frac{x}{t^H}\right)} \frac{e^{-\frac{z^2}{2}}}{\sqrt{2 \pi}}dz\vert \\ \leq & c|x-y| |t^{-H}|\\
 \leq & c|x-y| |t_0^{-H}|.
\end{align*}
By the triangle inequality we obtain the joint Hölder continuity
\begin{align*}
|\Phi^H_t(x)-\Phi^H_s(y)|=|\Phi^H_t(x)-\Phi^H_t(y)|+|\Phi^H_t(y)-\Phi^H_s(y)|\\\leq c \max\left(|t_0^{-H}|,\frac{|y|}{t_0^{2H}}\right) \left(|x-y| +  |t-s|^{H}\right).
\end{align*}
\end{Example}

\begin{Example}[Exponential Marginals and fBM copula]
Several modeling situations (e.g., when modelling stochastic volatility, interest rates, etc. in financial mathematics) necessitate positive stochastic processes. 
It is simple to see that copula constructions might lead to good interpretable and alternative methods to model such process, since we are free to put any continuous family of marginals onto a Gaussian process (this was for example suggested in \cite{WilsonGharamani2010}). 

As a simple example, take exponential marginals of the form
$$G_t(x):=\1_{x>0} \left(1-e^{-\frac x{t^H}}\right), \quad t\in [t_0,\infty), x\in\mathbb R$$
for some $t_0>0$, a (Hurst-)parameter $H=(0,1)$ corresponding to the copula process $(U_t)_{t\in T}$ of a fractional Brownian motion $B^H$ (we take the parameter $\frac 1{t^H}$ for the marginals to keep the same variance as the underlying fractional Brownian motion). By the smoothness of 
\begin{equation*}
G_t^{-1}(y)=-\log(1-y)t^h
\end{equation*}
we obtain that the transformed fractional Brownian motion 
\begin{equation*}
Y_t:=G_t^{-1}\left(\Phi_t^{H}(B_t^H)\right):=-\log\left(1-\int_{-\infty}^{B_t^H}\frac{e^{-\frac{z^2}{2t^{2H}}}}{\sqrt{2 \pi}t^H}dz\right)t^H=-\log\left(\int_{B_t^H}^{\infty}\frac{e^{-\frac{z^2}{2t^{2H}}}}{\sqrt{2 \pi}t^H}dz\right)t^H
\end{equation*}
has underlying Gaussian copula $U$, is $\gamma$-Hölder continuous for all $\gamma<H$ and has exponential marginals (with parameters $\frac 1{t^H}$). 

In \cite{Gatheral2018} it is argued empirically for lognormal marginals with a fractional Brownian motion copula for the stochastic volatility of asset prices. 
Our example shows that one can easily modify the marginals (to exponential, say, as in our example here), or other positively supported distributions, in so-called rough volatility models of asset prices. Moreover, the flexibility in the copula framework allows also to go beyond the specific dependency yielded by the copula induced by fractional Brownian motion.
\end{Example}

\subsubsection{Construction on Schauder Bases}\label{Sec:Schauder Basis Case}
In this section we will characterize
copula-constructed processes for random variables in  Banach spaces with a Schauder basis. This includes $L^p([0,1])$-spaces (with the Haar wavelets as Schauder basis), $C([0,1])$ (with the original Schauder basis), $l^p$-sequence spaces, and therefore in particular, all separable Hilbert spaces (with an orthonormal basis as Schauder basis). For a detailed account on the theory of bases in Banach spaces we refer to \cite{Heil2011}.
\begin{Definition}\label{Def: Schauder Basis}
A sequence $(e_n)_{n\in \mathbb{N}}\subseteq V$ of linearly independent vectors is called a basis of a locally convex Hausdorff space $V$, if for all $v\in V$ there is a unique sequence $a_n(v)$, such that 
$$v=\sum_{n\in \mathbb{N}} a_n(v) e_n,$$
where the series converges with respect to the locally convex topology on $V$. 
 A basis of $V^*$ is called  $weak^*$-basis of $V^*$, if it is a basis with respect to the $weak^*$-topology.
 If $V$ is a Banach space and $v\mapsto a_n(v)$ is continuous with respect to the norm topology 
for all $n\in\mathbb N$, we call
$(e_n)_{n\in \mathbb{N}}\subseteq V$ a Schauder basis.
\end{Definition}
The continuity of the function $(a_n)_{n\in \mathbb{N}}$ is automatically satisfied, if $E$ is a separable Banach space (see Theorem 3.1. in \cite{Singer1970}).
Note,
that every Banach space that possesses a basis is separable. However, for a separable Banach space, the existence of a basis cannot be guaranteed, due to the counterexample by Enflo in \cite{Enflo1973}. 
For a Banach space with Schauder basis we can verify, that the corresponding coefficient functions are always contained in the topological dual:
\begin{Lemma}
Let $V$ be a Banach space with Schauder basis $(e_n)_{n\in \mathbb{N}}$ and coefficient functions $(a_n)_{n\in \mathbb{N}}$. Then
$\lbrace a_n: n\in\mathbb N\rbrace\subset V^*$ and for $m,n\in\mathbb N$ we have
$$a_n(e_m)=\begin{cases}0 & m\neq n\\
1& m=n.\end{cases}$$
\end{Lemma}
\begin{proof}
Linearity of the coefficients is clear due to uniqueness of the representation. Moreover for the same reason, $a_n(e_n)=1$ and $a_m(e_n)=0$ gives a valid series representation of $e_n$ for all $n\in\mathbb N$ and by uniqueness of this, the assertion follows.
\end{proof}
That the coefficient functionals are linearly independent and separate the points is a consequence of the following Theorem:
\begin{Theorem}
Let $V$ be a Banach space.
A sequence $(a_n)_{n\in \mathbb{N}}$ is a  $weak^*$ Schauder basis of $V^*$ if and only if there exists a Schauder basis $(e_n)_{n\in\mathbb{N}}$ of $V$ that has $(a_n)_{n\in \mathbb{N}}$ as its coefficient functionals. The coefficient functionals for the basis $(a_n)_{n\in \mathbb{N}}$ are then given by the bidual elements $(\iota_{e_n})_{n\in\mathbb{N}}$, where $\iota_v(v^*)=v^*(v)$.
\end{Theorem}
\begin{proof}
See Theorem 14.1. in  \cite{Singer1970}.
\end{proof}
Observe that for a Banach space with Schauder basis, the corresponding set 
\begin{equation*}
    M=\lbrace a_n:n\in\mathbb N\rbrace
\end{equation*}
 of coefficient marginals satisfies all the conditions of Definition \ref{D: W-Marginals}. Embedding \eqref{Product space embedding} reads now
\begin{equation}\label{Basis embedding}
    v\mapsto (a_n(v))_{n\in\mathbb N}.
\end{equation}

\begin{Theorem}\label{T: Basis check up}
Let $V$ be a Banach space with Schauder basis $(e_n)_{n\in \mathbb{N}}$ and coefficient functions $(a_n)_{n\in \mathbb{N}}$. 
The following are equivalent
\begin{itemize}
    \item[(i)]A sequence $(a_n)_{n\in\mathbb N}$ is in the range of the embedding \eqref{Basis embedding};
    \item[(ii)]$\sum_{n=1}^{\infty}a_n e_n\in V$ is a convergent series in the norm topology;
    \item[(iii)] $\sup_{N\in\mathbb N}\|\sum_{n=1}^N a_n e_n\|<\infty$.
\end{itemize}
\end{Theorem}
\begin{proof}
See Theorem 4.13 in \cite{Heil2011}.
\end{proof}
Thus, for the checkup of the Construction \ref{abstract construction} we have the following:
\begin{Corollary}
Let $V$ be a Banach space with Schauder basis $(e_n)_{n\in\mathbb{N}}$ 
and $(X_n)_{n\in\mathbb N}$ be a stochastic process. 
Then the following are equivalent:
\begin{itemize}
\item[(i)] $X=\sum_{n=1}^{\infty} X_n e_n$ is a Borel measurable random variable in $V$;
\item[(ii)]\label{Norm condition for Schaudercopulas}
$\sup_{N\in\mathbb{N}}\|\sum_{n=1}^N X_n e_n\| < \infty \quad \mathbb{P}-\text{almost surely}$.
\end{itemize}
\end{Corollary}
\begin{proof}
This follows directly from Theorem \ref{T: Basis check up} and Lemma \ref{L: Measurability is not a problem in separable Banach spaces}.
\end{proof}
Let us now describe how we can construct a Banach space probability measure with predescribed dependence structure and marginals for the basis components:
\begin{Construction}\label{basis construction}\hspace{10cm}
\begin{enumerate}[label={(\roman{enumi})}]
\item\label{First Step in Basis Construction} $V$ a Banach space with  Schauder basis $(e_n)_{n\in \mathbb{N}}\subseteq V$;
\item\label{Second Step in Basis Construction} Choose a copula measure $C$ on $\mathbb{R}^{\mathbb N}$ (which models the dependency between basis elements) and marginals $(\mu_n)_{n\in\mathbb N}$.
 Merge them to a law  of a random sequence $(X_n)_{n\in\mathbb N}$  
 taking values in $\mathbb{R}^{\mathbb{N}}$ via Sklar's Theorem \ref{T: Sklar in infinite dimensions}.
\item\label{Third Step in Basis Construction} Define 
$X:=\sum_{n\in \mathbb N}X_n e_n.$
\item\label{Fourth Step in Basis Construction} Check if this sum converges in $V$ almost surely (corresponding to Corollary \ref{Norm condition for Schaudercopulas}).
\end{enumerate}
\end{Construction}
For the verification of \ref{Fourth Step in Basis Construction} we obtain conditions on the moments of the marginals.
\begin{Corollary}
Let $X$ be given as in Construction \ref{basis construction}\ref{Third Step in Basis Construction}. 
Then
 $X\in L^1(\Omega;V,\mathcal{B}(V))$ if the marginals have finite first moment and 
\begin{equation}\label{sufficient moment condition}
\sum_{n=1}^{\infty} \mathbb{E}[|X_n|]<\infty.
\end{equation}
\end{Corollary}
\begin{proof}
This follows immediately by using Corollary \ref{Norm condition for Schaudercopulas} and the triangular inequality.
\end{proof}
In case of sequence spaces, we obtain even sufficient and necessary conditions to construct laws with finite moments of certain order. Denote 
\begin{equation*}
l^p:=\left\lbrace (x_n)_{n\in\mathbb{N}}\subset \mathbb{R}^{\mathbb N}: \| (x_n)_{n\in\mathbb{N}}\|_p:=\left(\sum_{n=1}^{\infty} |x_n|^p\right)^{\frac 1p}<\infty\right\rbrace
\end{equation*}
for some $p\in [1,\infty)$.
\begin{Corollary}\label{C:lp moment condition}
 Let $V=l^p$ and $X$ be given as in Construction \ref{basis construction}\ref{Third Step in Basis Construction}. Then
 $X\in L^p(\Omega;l^p,\mathcal{B}(l^p))$   if and only if the marginals have finite $p$th moment and
\begin{equation}\label{moment condition for lp}
 \sum_{n=1}^{\infty} \mathbb{E}[|X_n|^p] < \infty.
\end{equation}
\end{Corollary}
\begin{proof}
The standard basis $(\delta_n)_{n\in\mathbb{N}}$ is the sequence which has components equal to zero everywhere, except on the $n$'th entry, where it is 1.
This defines a Schauder basis on $l^p$ with coefficient functionals $\delta_i^*$ given by $\delta_i^*((x_n)_{n\in\mathbb N})=x_i$, since
\begin{equation*}
(x_n)_{n\in\mathbb{N}}=\sum_{i=1}^{\infty} x_i \delta_i=\sum_{i=1}^{\infty} \delta_i^*((x_n)_{n\in\mathbb{N}}) \delta_i.
\end{equation*}
Thus,
\begin{align*}
\sup_{N\in\mathbb{N}}\left\|\sum_{n=1}^N X_n \delta_n\right\|_p^p=\sup_{N\in\mathbb{N}} \sum_{n=1}^N |X_n|^p= \sum_{n=1}^{\infty} |X_n|^p.
\end{align*}
This implies 
\begin{align*}
\mathbb{E}[\| X\|_p^p]=\sum_{n=1}^{\infty} \mathbb{E}[|X_n|^p]<\infty.
\end{align*}
The assertion follows.
\end{proof}

\begin{Remark}
Observe that Corollary \ref{C:lp moment condition} generalises Corollary 4.3 in \cite{MR3574701}, where the case of separable Hilbert spaces, that is $p=2$, was considered and the notion of a Schauder basis is reduced to the concept of orthonormal bases. 
\end{Remark}

The results derived above just impose conditions on the marginals, which makes them useful from a practical viewpoint. 
Still, the concept of copulas for random variables in the space $L^p(\Omega;l^p)$, or equivalent for laws in the Wasserstein space $\mathcal{W}_p(l^p)$ (see \eqref{Wasserstein space} below) is characterized completely by Corollary \ref{C:lp moment condition}. 
We will obtain another characterization of copulas as underlying solutions to certain restricted optimization problems in these Wasserstein spaces in the next section.

\section{Robustness of the Copula Construction}\label{sec: Robustness of the copula construction}
The previous section suggested that copula theory is well suited for the spaces $\mathcal L^p(T):=\mathcal L^p(T,\mathcal B (T),\mu;\mathbb R)$ for a finite Borel measure $\mu$ and $T\subset \mathbb{R}^d$ a compact interval
and the sequence spaces $l^p$, due to simple moment criteria to overcome the construction problem.
In this section we will provide distance estimates of random variables in these  
spaces in terms of their copula and marginals separately.

Hereafter we shorten the notation as follows: 
For a random variable  $X$ with values in $E$ (where $E$ equals $l^p$ or $\mathcal{L}^p(T)$ respectively) we denote the operators 
\begin{equation*}
F_X(x)_n:=F_{X_n}(x_n),\,\,\, n\in \mathbb N \quad  (    F_X(f)(t):=F_{X_t}(f(t)),\,\,\,t\in T \,\,\text{ respectively})
\end{equation*}
and
\begin{equation*}
 F_X^{[-1]}(x)_n:=F_{X_n}^{[-1]}(x_n),\,\,\, n\in \mathbb N \quad (  F^{[-1]}_X(f)(t):=F_{X_t}^{[-1]}(f(t)),\,\,\,t\in T  \,\,\text{ respectively})
\end{equation*}
for all $x\in l^p$ (and $f\in\mathcal{L}^p(T)$ respectively). Moreover, we use the notation $U^X$ for the underlying copula process of $X$.
We will for convenience switch between the spaces $\mathcal{L}^p(T)$ and $L^p(T)$ whenever there is no confusion. If we say that an $[X]\in L^p(T)$ has underlying copula process $U\in \mathcal L^p(T)$, we mean that there is a representative $X\in \mathcal L^p(T)$ of the corresponding element, that has this path copula. We will also drop equivalence class notation from time to time, to ease the writing (especially, when we work with Wasserstein spaces in the next section) and just refer to the representative $X$, no matter if we mean the equivalence class or the actual stochastic process.

\subsection{Copulas and Wasserstein Spaces}\label{sec: Copulas and Wasserstein spaces}
In this subsection we characterize copulas for measures in Wasserstein spaces. 
For two laws $\nu^1$ and $\nu^2$ on $E$ we write $\rho<^{\nu^2}_{\nu^1}$ for a law $\rho$ on $E\times E$ that has marginal distributions $\nu^1$ and $\nu^2$, that is $\rho$ is a coupling of $\nu^1$ and $\nu^2$.
Recall that the $p$-Wasserstein space over a separable Banach space $E$ is a complete separable metric space (see e.g.
 \cite{Villani2008}) given by
\begin{equation}\label{Wasserstein space}
    \mathcal{W}_p(E):=\left\lbrace \nu: \nu \text{ is a Borel law on }E,\,\int_{E} \| x\|_E^p \nu(dx) <\infty\right\rbrace
\end{equation}
equipped with the metric (in the case that we interpret $E=L^p(T)$ instead of $E=\mathcal L^p(T)$)
\begin{equation*}
    d(\nu^1,\nu^2)=:\mathbb{W}_p(\nu^1,\nu^2):=\inf_{\rho<^{\nu^2}_{\nu^1}}\left(\int_{E\times E} \|x-y\|_E^p\rho(dx dy)\right)^{\frac 1p}.
\end{equation*}
If there are two random variables $X\sim \nu^1$ and $Y\sim \nu^2$, we also say that $(X,Y)$ is a coupling and write $\mathbb{W}_p(\nu^1,\nu^2)=\mathbb{W}_p(X,Y)$.
If $E=\mathbb R$, we have the following closed form of the Wasserstein distance (see e.g. Theorem 3.1.2 in \cite{Rachev1998}):
\begin{equation}\label{Closed form of the Wasserstein distance in one dimension}
    \mathbb{W}_p^p(X,Y)=\int_{[0,1]}|F_X^{[-1]}(u)-F_Y^{[-1]}(u)|^p du.
\end{equation}
\begin{Theorem}\label{T: Wasserstein-Copula Best Approximation}
Let $X,Y$ be random variables in $l^p$ (in $\mathcal{L}^p(T)$ respectively) for some $p\in\mathbb N$. Then the following are equivalent:
\begin{itemize}
    \item[(i)] $X$ and $Y$ share the same underlying basis copula (path copula respectively) $C$;
    \item[(ii)]$(F_X^{[-1]}(U),F_Y^{[-1]}(U))$ is an optimal coupling of $X$ and $Y$, where $U\sim C$;
    \item[(iii)] The Wasserstein distance between $X$ and $Y$ is given by
    \begin{align*}
         \mathbb{W}_p^p(X,Y)=\sum_{n\in\mathbb N} \mathbb{W}_p^p(X_n,Y_n)\\
         (\text{respectively } \mathbb{W}_p^p(X,Y)=\int_T\mathbb{W}_p^p(X_t,Y_t)\mu(dt)).
    \end{align*}
\end{itemize}
In particular, if one of the above holds we have
    \begin{align*}
    \sum_{n\in\mathbb N} \mathbb{W}_p^p(X_n,Y_n)=\|F_X^{[-1]}(U)-F_Y^{[-1]}(U)\|_{L^p(\Omega;l^p)}^p\\
    (\text{respectively } \int_T\mathbb{W}_p^p(X_t,Y_t)\mu(dt)=\|F_X^{[-1]}(U)-F_Y^{[-1]}(U)\|_{L^p(\Omega;L^p(T))}^p).
      \end{align*}
\end{Theorem}
\begin{proof}
Since the proof for the $\mathcal{L}^p(T)$ case is analogous, we will just show the assertion for $l^p$ valued random variables $X$ and $Y$.

Assume $(i)$ holds.
By Corollary \ref{C:lp moment condition} we have that $F_X^{[-1]}(U)$ and $F_X^{[-1]}(U)$ are measurable random variables taking values in $l^p$ for a copula process $U\sim C$. Moreover, they are a coupling, as consequence of Sklar's Theorem \ref{T: Sklar in infinite dimensions}. 
To show optimality, observe first that for $X\sim \nu^1$ and $Y\sim \nu^2$ we have
\begin{align}\label{Wasserstein Lower bound}
\mathbb{W}_p^p(X,Y)= & \inf_{\rho<_{\nu^2}^{\nu^1}}\int_{l^p\times l^p} \| x-y\|_p^p \rho(dx,dy)\notag\\
= & \inf_{\rho<_{\nu^2}^{\nu^1}}\int_{l^p\times l^p} \sum_{i=1}^{\infty} | x_i-y_i|^p \rho(dx,dy)\notag\\
\geq & \sum_{i=1}^{\infty} \inf_{\rho<_{\nu^2}^{\nu^1}}\int_{l^p\times l^p}  | x_i-y_i|^p \rho(dx,dy)\\
= & \sum_{i=1}^{\infty} \inf_{\rho_i<_{\nu^2_i}^{\nu^1_i}}\int_{\mathbb R\times \mathbb R}  | x_i-y_i|^p \rho_i(dx_i,dy_i)\notag\\
= & \sum_{i=1}^{\infty} \mathbb{W}_p^p(\nu^1_i,\nu^2_i)=\sum_{i=1}^{\infty} \mathbb{W}_p^p(X_i,Y_i).\notag
\end{align}
This general lower bound on the Wasserstein distance is actually achieved in our case since, by \eqref{Closed form of the Wasserstein distance in one dimension}, we obtain
\begin{align*}
\sum_{i=1}^{\infty} \mathbb{W}_p^p(X_i,Y_i)
=  &  \sum_{i=1}^{\infty} \int_{[0,1]}  | F_{X_i}^{[-1]}(u_i)-F_{Y_i}^{[-1]}(u_i)|^p du_i\\
 = &   \int_{[0,1]^{\mathbb N}}  \| (F_{X}^{[-1]}(u)-F_{Y}^{[-1]}(u))\|_p^p C(du)\\
 = &   \|F_X^{[-1]}(U)-F_Y^{[-1]}(U)\|_{L^p(\Omega;l^p)}^p\\
 \geq &  \mathbb{W}_p^p(X,Y).
\end{align*}
This shows $(i)\Leftrightarrow(ii)$ and $(i)\Rightarrow (iii)$.
Since $(ii)\Rightarrow (i)$ is trivial, it is therefore sufficient to show $(iii)\Rightarrow (i)$. 
Since equality in \eqref{Wasserstein Lower bound} can just hold, if there is an optimal coupling $(X,Y)$, such that $\mathcal{W}_p^p(X_i,Y_i)=\mathbb E[|X_i-Y_i|^p]$, we have that $(X_i,Y_i)$ must also be an optimal coupling for all $i\in\mathbb N$. By Proposition 2.1 in \cite{Ruschendorf2009} we obtain that for all $i\in \mathbb N$ we have that $X_i\overset{\text{s.o.}}{\sim}Y_i$.
  This implies $(i)$ due to  Lemma \ref{Lem: Same copulas means similarly ordered}.
\end{proof}
\begin{Remark}
Observe that the implications $(ii)\Rightarrow (i)$ and $(iii)\Rightarrow (i)$ in Theorem
\ref{T: Wasserstein-Copula Best Approximation} must be interpreted in the sense that there is always a representative of the equivalence classes that possesses the same path copula.
\end{Remark}
\begin{Remark}
An analogous result to the previous Theorem \ref{T: Wasserstein-Copula Best Approximation} was elaborated in finite dimensions in \cite{Ruschendorf2009} and then transferred to the infinite-dimensional setting via the equivalent formulation for two random variables with similarly ordered marginals (which, by Lemma \ref{Lem: Same copulas means similarly ordered}, is equivalent for them to share the same copula).
 However, since there the proof was not given explicitly and the notion of copulas in infinite dimensions was not used, we provided a proof.

Furthermore, the assertion of Theorem \ref{T: Wasserstein-Copula Best Approximation} does not hold for the $q$-Wasserstein distance over $l^p$ ($\mathcal L^p(T)$ respectively) if $q\neq p$, as it was shown in \cite{alfonsi2014} for the finite-dimensional case.
\end{Remark}
\begin{Remark}
Theorem \ref{T: Wasserstein-Copula Best Approximation} is useful, because the one-dimensional Wasserstein distance has a closed form given by \eqref{Closed form of the Wasserstein distance in one dimension}. This expression can oftentimes be estimated rather well from above 
(see for instance chapter 4.7 in \cite{Panaretos2020} for a discussion of convergence of empirical measures).
\end{Remark}
\begin{Remark}
The copula construction effectively solves the following optimization problem for $E=l^p$ or $E=\mathcal L^p(T)$ respectively:
\begin{align*}
(P)=\begin{cases}
\min_{\nu\in \mathcal W_p(E)} \mathbb{W}_p(\nu^0,\nu)\\
\text{s.t.}\quad \nu\text{ has marginals }(\nu_n)_{n\in \mathbb N}\,\, (\text{respectively }(\nu_t)_{t\in T})
\end{cases}
\end{align*} 
for any family of marginals $(\nu_n)_{n\in \mathbb N}$ $(\text{respectively }(\nu_t)_{t\in T})$, and the optimal value is given by $\nu=((F_{\nu_n}^{[-1]})_{n\in \mathbb N})_*C$ (respectively $\nu=((F_{\nu_t}^{[-1]})_{t\in T})_*C$) for the underlying copula measure $C$ of $\nu^0$. 
\end{Remark}
Moreover, Theorem \ref{T: Wasserstein-Copula Best Approximation} implies the following.
\begin{Corollary}\label{Lem: Copula Conservation law}
Let $X,Y$ be stochastic processes with values in $ \mathcal{L}^2(T)$. If $(e_n)_{n\in\mathbb N}$ is an orthonormal basis in $L^2(T)$, then the following are equivalent:
 \begin{itemize}
     \item[(a)] $\int_T \mathbb{W}_2^2(X_t,Y_t)dt=\sum_{n=1}^{\infty}\mathbb{W}_2^2(\langle X,e_n\rangle_{L^2(T)},\langle Y,e_n\rangle_{L^2(T)})$
     \item[(b)] $X$ and $Y$ have the same basis copula if and only if $X$ and $Y$ have the same path copula.
     \end{itemize}
\end{Corollary}

\subsection{A Robustness Inequality in $\mathcal L^p(T)$}\label{sec: Lp estimate}
In order to derive a distance estimate between random variables in  $\mathcal L^p(T):=\mathcal L^p(T,\mathcal B (T),\mu;\mathbb R)$ 
for a finite Borel measure $\mu$ and $T\subset \mathbb{R}^d$ compact, based on the copula and the marginals separately, we impose a smoothness assumption on marginals of the distribution function.
\begin{Assumption}\label{As: Smoothness and Tail condition}
For all $t$, the marginals $F_t$ are continuously differentiable and strictly increasing on $(F_t^{[-1]}(0+),F_t^{[-1]}(1))$. Moreover, assume that for the corresponding densities $f_t$
there is a measurable 
    \begin{equation*}
         g:T\times \mathbb R\to \mathbb R,\,\, g(t,x)=:g_t(x)
    \end{equation*}
    such that each
    $g_t$  is ultimately monotone (see, e.g. \cite{Bingham1987}), that is, it is monotone on $[m_t-x_0^t,m_t+x_0^t]^c$ for some $x_0^t\in \mathbb{R}_+$, $m_t\in\mathbb R$, with $g_t$ bounded away from $0$ on $ [m_t-x_0^t,m_t+x_0^t]$ by some $\lambda>0$ (independent of $t$) and $f_t(x)\geq g_t(x)>0$ for all $x\in (F_{Y_t}^{[-1]}(0),F_{Y_t}^{[-1]}(1))$ and for all $t\in T$, and there is an $0<\beta\leq 1$ such that
\begin{align}\label{Tail condition}
 \int_T  \int_{\mathbb R}\frac {f_t(x)}{g_t^{\beta}(x)}dxdt <\infty.
 \end{align}
\end{Assumption}
\begin{Remark}
If a density function $(x,t)\mapsto f_t(x)$ is continuous and ultimately monotone, that is, there are $x_0^t>0$, $m_t\in\mathbb R$ such that $f_t$ is monotone on $[m_t-x_0^t,m_t+x_0^t]^c$, the best candidate for the choice of $g$ in Assumption \ref{As: Smoothness and Tail condition} is $f_t$ itself.
\end{Remark}
 Observe that Assumption \ref{As: Smoothness and Tail condition} is satisfied for Gaussian marginals:
\begin{Example}
Assume that $Y=W$ is a zero mean continuous Gaussian process. Clearly, the densities $f_t$ of $W_t$ are ultimately monotone and and we can choose $x_0^t=0$ 
and $g_t=f_t$.
Then Condition \ref{Tail condition} holds, as for all $\beta\in (0,1)$ 
\begin{align*}
\int_T\int_{\mathbb R}f_t^{1-\beta}(x)dxdt = & \int_T\frac{\mathbb E [e^{\beta\frac{W_t^2}{2\sigma_t^2}}]}{(\sqrt{2\pi}\sigma_t)^{-\beta}}dt = & (\sqrt{2\pi})^{\beta}\int_T\sigma_t^{\beta}dt  \mathbb E [e^{\beta\frac{Z^2}{2}}]=\frac {(\sqrt{2\pi})^{\beta}}{\sqrt{1-\beta}} \int_T\sigma_t^{\beta}dt.
\end{align*}
where $Z$ is a standard normally distributed random variable. Thus, since $t\mapsto \sigma_t^{\beta}$ is continuous, it is integrable over $T$ and Assumption \ref{As: Smoothness and Tail condition} holds.
\end{Example}
Another example, for which Assumption \ref{As: Smoothness and Tail condition} holds, is the following class of heavy tailed marginals.
\begin{Example}\label{Ex: Regularly varying marginals}
(Regularly varying marginals)
A measurable  function $h:\mathbb{R}_+\to\mathbb{R}_+$ is regularly varying with tail index $\alpha\in\mathbb R$, if
\begin{equation*}
    \lim_{x\to\infty}\frac{h(tx)}{h(x)}=t^{\alpha}
\end{equation*}
We write $h\in\mathcal R (\alpha)$. If $\alpha=0$, $h$ is called slowly varying. A one-dimensional law given by its cumulative distribution function $F$ is said to have regularly varying tails, if the survival function $\bar{F}:=1-F$ is regularly varying.

Let $Y$ be a c{\`a}dl{\`a}g stochastic process, such that its marginals $(F_t)_{t\in T}$ are continuously differentiable, strictly increasing, supported on $\mathbb R_+$ and regularly varying with tail index $-\alpha_t$ for $\alpha_t >0$ where we assume that $t\mapsto \alpha_t$ is continuous. Moreover, assume that the densities $f_t$ are ultimately monotone on $[0,y_0^t]^c$ for some $y_0^t\in\mathbb{R}_+$ and jointly continuous in $x$ and $t$. 
This enables us to use the monotone density theorem (c.f. Theorem 1.7.2 in \cite{Bingham1987}) to conclude that $f_t\in \mathcal R (-(\alpha_t+1))$. Hence, there are slowly varying functions $l_t:\mathbb{R}_+\to\mathbb{R}_+$ such that \begin{equation*}
    f_t(x)=x^{-(1+\alpha_t)}l_t(x).
\end{equation*}
For convenience, let us assume that $(t,x)\mapsto l_t(x)$ is bounded.
 By choosing $\beta<\min_{t\in T}\frac{\alpha_t}{1+\alpha_t}$ we obtain  
 \begin{equation*}
    l_t^{1-\beta}(x) x^{-(1-\beta)(1+\alpha_t)}\in \mathcal R (- (1-\beta)(1+\alpha_t))
\end{equation*} 
such that $- (1-\beta)(1+\alpha_t)<1$ and by Karamata's theorem (c.f. Proposition 1.5.10 in \cite{Bingham1987})
we can find $x_0^t>y_0^t$ and some $\delta>0$ such that
\begin{align*}
    \int_{x_0^t}^{\infty} \frac{f_t(x)}{f^{\beta}_t(x)}dx= \int_{x_0^t}^{\infty} f_t^{1-\beta}(x)dx = & \int_{x_0^t}^{\infty} l_t^{1-\beta}(x) x^{-(1-\beta)(1+\alpha_t)}dx\notag\\
    \leq & (1+\delta) l_t^{1-\beta}(x_0^t)(x_0^t)^{-(1-\beta)(1+\alpha_t)+1}<\infty.
\end{align*}
Assume moreover that
we can choose $t\mapsto x_0^t$ to be continuous (this is possible for instance if all $l_t$'s are supported on a compact domain) and hence, since $f_t(x)>0$ for all $x\in\mathbb R_+$ we have for
\begin{equation*}
    \lambda:=\min_{t\in T}\min_{ x\in[-x_0^t,x_0^t]}f_t(x)>0
\end{equation*}
 that each $f_t$ is bounded away from $0$ on $ [-x_0^t,x_0^t]^c$ by this $\lambda >0$. Moreover, it holds
 \begin{equation*}
     \int_T \int_{x_0^t}^{\infty}\frac{f_t(x)}{f^{\beta}_t(x)}dxdt\leq \int_T (1+\delta) l_t^{1-\beta}(x_0^t)(x_0^t)^{-(1-\beta)(1+\alpha_t)+1} dt<\infty
 \end{equation*}
by the continuity of $t\mapsto x_0^t$.
Thus, \eqref{Tail condition} in Assumption \ref{As: Smoothness and Tail condition} is valid with 
\begin{align}\label{Reg.var. Marginals minorant}
   g_t(x):=\begin{cases} 0 & x<0\\
   \lambda & 0\leq x<x_0^t\\
f_t(x) & x\geq x_0^t.
\end{cases} 
\end{align}
\end{Example}
The next Theorem gives an idea about the robustness of the copula construction.
\begin{Theorem}\label{Thm:Robustness inequalty}
Let $X=(F_{X_t}^{[-1]}(U_t^X))_{t\in T}$ and $Y=(F_{Y_t}^{[-1]}(U_t^Y))_{t\in T}$ be c{\`a}dl{\`a}g stochastic processes, such that $[X]\in L^p(\Omega\times T)$ for some $p\geq 1$, $[Y]\in L^{p+\epsilon}(\Omega\times T)$ for some $\epsilon>0$ and let the marginals $F_Y$ of $Y$ satisfy Assumption \ref{As: Smoothness and Tail condition}.
Then for all $q\geq 1$ there are constants $K:=K(\beta,q,p,\epsilon,F_Y)$ and $\rho:=\rho(\beta,q,p,\epsilon)$
 such that 
     \begin{align}\label{Copula robustness estimate}
        \|X-Y\|_{L^p(\Omega\times T)}^p\leq \|\mathbb{W}_p(F_{X_{\cdot}},F_{Y_{\cdot}})\|_{L^p(T)}+  K \|U^X-U^Y\|_{L^q(\Omega\times T)}^{\rho}.
     \end{align}
     The constants are given by 
\begin{equation}
    \rho:= \frac{\epsilon q\beta}{p(p+\epsilon)(q+\beta)- pq\beta}
\end{equation}
 and
 \begin{equation}\label{Constant K}
  K:= \left(\lambda^{-\beta}\int_T\1_{(0,\infty)}(x_0^t)dt+2\|  g_t^{-\beta}(Y_t)\|_{L^1(\Omega\times T)}\right)^{\frac{\rho }{\beta}} (2 \|Y\|_{L^{p+\epsilon}(\Omega\times T)})^{(1-\rho)
  }
\end{equation}
\end{Theorem}

\begin{proof}
By the triangle inequality we have
\begin{equation}\label{Disentangeling Theorem: Equation 1}
    \mathbb E \left[\|X-Y\|_{L^p(T)}^p\right]^{\frac 1p}\leq \mathbb E\left[\|X-F_Y^{[-1]}(U_X)\|_{L^p(T)}^p\right]^{\frac 1p}+E\left[\|F_Y^{[-1]}(U_X)-Y\|_{L^p(T)}^p\right]^{\frac 1p}.
\end{equation}
From Theorem \ref{T: Wasserstein-Copula Best Approximation} we know that $(X,F_Y^{[-1]}(U_X))$ is an optimal coupling and
\begin{equation}\label{Disentangeling Theorem: Equation 2}
   \mathbb E\left[\|X-F_Y^{[-1]}(U_X)\|_{L^p(T)}^p\right]^{\frac 1p}=\|\mathbb{W}_p(F_{X_{\cdot}},F_{Y_{\cdot}})\|_{L^p(T)}.
\end{equation}
Let us now estimate the second summand. 
Set $\delta:= 1+\frac{p(q+\beta)-q\beta}{(q+\beta)\epsilon} > 1 $, such that $\frac{\delta-\frac {q\beta}{(q+\beta)p}}{\delta-1}=1+\frac{\epsilon}{p}$. Then we can estimate for $\gamma =\frac {\delta}{\delta-1}$ using Hölder's inequality 
\begin{align}\label{Disentangeling Theorem: Equation 3}
   & \mathbb E\left[\|F_Y^{[-1]}(U_X)-Y\|_{L^p(T)}^p\right]\notag\\
   = & \int_T \mathbb E \left[|F_{Y_t}^{[-1]}(U^X_t)-Y_t|^p\right] dt\notag\\
    = & \int_T \mathbb E \left[|F_{Y_t}^{[-1]}1(U^X_t)-Y_t|^{\frac{q\beta}{(q+\beta)\delta}} |F_{Y_t}^{[-1]}(U^X_t)-Y_t|^{p-\frac{q\beta}{(q+\beta)\delta}}\right] dt\notag\\
     \leq & (\int_T \mathbb E \left[|F_{Y_t}^{[-1]}(U^X_t)-Y_t|^{\frac{q\beta}{(q+\beta)}} \right]dt)^{\frac 1{\delta}}(\int_T \mathbb E\left[|F_{Y_t}^{[-1]}(U^X_t)-Y_t|^{\gamma(p-\frac{q\beta}{(q+\beta)\delta})}\right] dt)^{\frac 1{\gamma}}.
     \end{align}
     
    Now observe that since $Y_t$ and $F_{Y_t}^{[-1]}(U_t^X)$ share the same distribution and by the elementary inequality $|x+y|^r\leq 2^{r-1}(|x|^r+|y|^r)$ for $r\geq 1$ we have
    \begin{align}\label{Disentangeling Theorem: Equation 4}
        \int_T \mathbb E \left[|F_{Y_t}(U^X_t)-Y_t|^{\gamma(p-\frac{q\beta}{(q+\beta)\delta})}\right] dt=&  \int_T \mathbb E \left[|F_{Y_t}(U^X_t)-Y_t|^{p\frac{\delta-\frac{q\beta}{(q+\beta)p}}{\delta-1}}\right] dt\notag\\
       = &  \int_T \mathbb E \left[|F_{Y_t}(U^X_t)-Y_t|^{p+\epsilon}\right] dt\notag\\
        \leq& 2^{p+\epsilon-1}\int_T \mathbb E\left[|F_{Y_t}(U^X_t)|^{p+\epsilon}+|Y_t|^{p+\epsilon}\right] dt\notag\\
        = & 2^{p+\epsilon}\int_T \mathbb E\left[|Y_t|^{p+\epsilon}\right] dt\notag\\
        = & (2 \|Y\|_{L^{p+\epsilon}(\Omega\times T)})^{p
    +\epsilon}
    \end{align}
This shows that $ \int_T \mathbb E [|F_{Y_t}(U^X_t)-Y_t|^{\gamma(p-\frac{q\beta}{(q+\beta)\delta})}] dt<\infty$, 
since $Y$ is assumed to have finite moments up to $p+\epsilon$.

Now observe that 
\begin{align}\label{Common extreme values are nullset}
    \mathbb P\left[(U_t^X,U_t^Y)\in (0,1)\right]= & 1-\mathbb P\left[U_t^X\in \lbrace 0,1\rbrace\text{ or }U_t^Y\in \lbrace 0,1\rbrace\right]\notag\\
  \geq  & 1-(\mathbb P\left[U_t^X\in \lbrace 0,1\rbrace\right]+\mathbb P\left[U_t^Y\in \lbrace 0,1\rbrace\right])
    =  1-0 = 1
\end{align}
Moreover, by Assumption \ref{As: Smoothness and Tail condition} we further have $
   0< f_{Y_t}(F_{Y_t}^{[-1]}(\zeta))$ for $\zeta \in (0,1)$.
Hence, since by \eqref{Common extreme values are nullset} $[\min(U_t^X,U_t^Y),\max(U_t^X,U_t^Y)]\subset (0,1)$ almost surely, we obtain by the inverse function theorem for $\zeta \in[\min(U_t^X,U_t^Y),\max(U_t^X,U_t^Y)]$
\begin{equation*}
   \frac d{dx} F_{Y_t}^{[-1]}(\zeta))=\left(f_{Y_t}^{-1}(F_{Y_t}(\zeta))\right)^{-1}.
\end{equation*}
Appealing to the mean value theorem and once more Hölder's inequality ($\|fg\|_{L^1(T)}\leq \|f\|_{L^{\frac{r}{r-1}}(T)}\|g\|_{L^r(T)}$ with $r=\frac {(q+\beta)}{\beta}$)  we obtain
\begin{align}\label{Disentangeling Theorem: Equation 5}
    &\int_T \mathbb E \left[|F_{Y_t}^{[-1]}(U^X_t)-Y_t|^{\frac{q\beta}{(q+\beta)}} \right]dt\notag\\
    \leq &  \int_T \mathbb E \left[\left(\sup_{\zeta\in\left[\min(U_t^X,U_t^Y),\max(U_t^X,U_t^Y)\right]}\left(f_{Y_t}\left(F_{Y_t}^{[-1]}(\zeta)\right)\right)^{-1}|U^X_t-U^Y_t|\right)^{\frac{q\beta}{(q+\beta)}} \right]dt\notag\\
     \leq &  \left(\int_T \mathbb E \left[\sup_{\zeta\in[\min(U_t^X,U_t^Y),\max(U_t^X,U_t^Y)]}\left(f_{Y_t}\left(F_{Y_t}^{[-1]}(\zeta)\right)\right)^{-\beta}\right]dt\right)^{\frac q{\beta+q}}\notag\\
     & \qquad \times\left(\int_T\mathbb E \left[|U^X_t-U_t^Y|^q \right]dt\right)^{\frac{\beta}{\beta+q}}.
\end{align}
     We now show that the first factor is finite.
     Denote the random variables
     \begin{align*}
         Z:=\max_{\zeta\in[\min(U_t^X,U_t^Y),\max(U_t^X,U_t^Y)]}\left(f_{Y_t}\left(F_{Y_t}^{[-1]}(\zeta)\right)\right)^{-\beta}\\
         \zeta^*=\argmax_{\zeta\in[\min(U_t^X,U_t^Y),\max(U_t^X,U_t^Y)]}\left(f_{Y_t}\left(F_{Y_t}^{[-1]}(\zeta)\right)\right)^{-\beta}
     \end{align*}
    and choose $x_0^t$ according to Assumption \ref{As: Smoothness and Tail condition} such that $g_t$ is ultimately monotone on $[-x_0^t,x_0^t]^c$, where without loss of generality $m_t=0$.
     We can argue by continuity and monotonicity of cumulative distribution and quantile functions as well as Assumption \ref{As: Smoothness and Tail condition} that
     \begin{align}\label{Estimating the inner part of the quantile slope}
         \mathbb E \left[\1_{\zeta^*\in (F_{Y_t}(-x_0^t),F_{Y_t}(x_0^t))} Z\right] \leq &  \sup_{\zeta\in (F_{Y_t}(-x_0^t),F_{Y_t}(x_0^t)) }\left(f_{Y_t}\left(F_{Y_t}^{[-1]}(\zeta)\right)\right)^{-\beta}\notag\\
         \leq &  \sup_{\zeta\in(F_{Y_t}(-x_0^t),F_{Y_t}(x_0^t)) }\left(g_t\left(F_{Y_t}^{[-1]}(\zeta)\right)\right)^{-\beta}\notag\\
          \leq & \1_{(0,\infty)}(x_0^t) \lambda^{-\beta}
     \end{align}
        Without loss of generality we can assume $g$ to be symmetric in the tails, that is $g(x)=g(-x)$ for $x\geq x_0$. For $\zeta^*\notin [F_{Y_t}(-x_0^t),F_{Y_t}(x_0^t)]$ we have by definition 
        \begin{equation*}
        \left[\min(U_t^X,U_t^Y),\max(U_t^X,U_t^Y)\right]\not\subset \left[F_{Y_t}(-x_0^t),F_{Y_t}(x_0^t)\right]
        \end{equation*}
        and thus, we must have either $U_t^X\in [F_{Y_t}(-x_0^t),F_{Y_t}(x_0^t)]^c$ or $U_t^Y\in [F_{Y_t}(-x_0^t),F_{Y_t}(x_0^t)]^c$. Hence, by Assumption \ref{As: Smoothness and Tail condition} as well as the monotonicity and the symmetry of $g$, we have
        \begin{align}\label{Final Monotone Assumption for slope of marginals}
        \1_{\zeta^*\notin [F_{Y_t}(-x_0^t),F_{Y_t}(x_0^t)]} Z\leq &  \1_{\zeta^*\notin [F_{Y_t}(-x_0^t),F_{Y_t}(x_0^t)]}\max_{\zeta\in [\min(U_t^X,U_t^Y),\max(U_t^X,U_t^Y)]}\left(g_t\left(F_{Y_t}^{[-1]}(\zeta)\right)\right)^{-\beta}\notag\\
        \leq & \max\left(\left(g\left(F_{Y_t}^{[-1]}\left(U_t^Y\right)\right)\right)^{-\beta},\left(g\left(F_{Y_t}^{[-1]}\left(U_t^X\right)\right)\right)^{-\beta}\right)
        \end{align}
      Therefore, for any uniformly distributed $U$ on $[0,1]$ we obtain
     \begin{align}\label{Estimating the tail part of the quantile slope}
            \mathbb E \left[\1_{\zeta^*\notin \left[F_{Y_t}(-x_0^t),F_{Y_t}(x_0^t)\right]} Z\right]
          \leq &  \mathbb E \left[  \max\left(\left(g_t\left(F_{Y_t}^{[-1]}\left(U_t^X\right)\right)\right)^{-\beta},\left(g_t\left(F_{Y_t}^{[-1]}\left(U_t^Y\right)\right)\right)^{-\beta}\right)\right]\notag\\
          \leq & 2\mathbb E \left[  \left(g_t\left(F_{Y_t}^{[-1]}(U)\right)\right)^{-\beta}\right]\notag\\
          = & 2\mathbb E \left[  \left(g_t(Y_t)\right)^{-\beta}\right].
     \end{align}
     Thus, \eqref{Estimating the inner part of the quantile slope} and \eqref{Estimating the tail part of the quantile slope} imply
     \begin{align}\label{Global slope estimate of quantile function}
       &  \left(\int_T \mathbb E \left[\sup_{\zeta\in[\min(U_t^X,U_t^Y),\max(U_t^X,U_t^Y)]}\left(f_{Y_t}\left(F_{Y_t}^{[-1]}(\zeta)\right)\right)^{-\beta}\right]dt\right)^{\frac q{\beta+q}}\notag\\
         \leq & \left( \lambda^{-\beta}\int_T\1_{(0,\infty)}(x_0^t)dt  +2\int_T\mathbb E \left[  \left(g_t(Y_t)\right)^{-\beta}\right]dt\right)^{\frac q{\beta+q}}
     \end{align}
     Combining \eqref{Disentangeling Theorem: Equation 3}, \eqref{Disentangeling Theorem: Equation 4}, \eqref{Disentangeling Theorem: Equation 5} and \eqref{Global slope estimate of quantile function} we obtain 
     \begin{align*}\label{Disentangeling Theorem: Equation 7}
         & \mathbb E\left[\|F_Y^{[-1]}(U_X)-Y\|_{L^p(T)}^p\right] \notag\\
         \leq &\left(\lambda^{-\beta}\int_T\1_{(0,\infty)}(x_0^t)dt+2\| \left(g_t(Y_t)\right)^{-\beta}\|_{L^1(\Omega\times T)}\right)^{\frac q{\delta\beta+\delta q}} \left(2 \|Y\|_{L^{p+\epsilon}(\Omega\times T)}\right)^{\frac {p
    +\epsilon}{\gamma}}\\
    & \qquad \times\|U^X-U^Y\|_{L^q(\Omega\times T)}^{\frac{q\beta}{\delta\beta+\delta q}}\notag\\
     = & K^p\|U^X-U^Y\|_{L^q(\Omega\times T)}^{\frac{\epsilon q\beta}{(p+\epsilon)(q+\beta)- q\beta}}.
     \end{align*}
The proof is complete. 
\end{proof}

\begin{Remark}
Although the marginals of $Y$ must fulfill Assumption \ref{As: Smoothness and Tail condition}, the marginals of $X$ can be chosen more freely and neither have to be absolutely continuous, nor must satisfy a tail condition. For instance, this allows to approximate a smooth law of $Y$ with discrete marginal measures (e.g. empirical measures).
\end{Remark}

\begin{Remark}
The parameters $q,p,\beta$ and $\epsilon$, used in 
equation \eqref{Copula robustness estimate} are competing in the following way: Choosing a lower $p$, but larger $\epsilon$ makes a potential convergence rate better, since it makes the exponent $\rho$ 
decrease (However, this also lets the constant $K$ grow larger). For the same reason we might wish to choose the largest value possible for $\beta$. 
The parameter $q$ can be chosen in order to derive a good approximation rate for the copula processes (for instance via the next Theorem \ref{Thm: Difference of copulas that stem from other processes}).
\end{Remark}
The next Theorem is useful if the copula processes stem from other processes, like for instance Gaussian or elliptical copulas.
\begin{Theorem}\label{Thm: Difference of copulas that stem from other processes}
Let $F_{\tilde Y},F_{\tilde X}$ be marginals with finite $q$th moment for $q\geq 1$ and define $\tilde X:= F_{\tilde X}^{[-1]}(U^X)$ and $\tilde Y:= F_{\tilde Y}^{[-1]}(U^Y)$.
Asume $F_{\tilde{Y}_t}$ is absolutely continuous, strictly increasing, and the corresponding density function is bounded, that is
\begin{equation*}
    \|f_{\tilde{Y}} \|_{\infty}:= \sup_{t\in T,x\in \mathbb R} |f_{\tilde{Y}_t}(x)|<\infty.
\end{equation*}
Then 
\begin{align*}
    \|U^X-U^Y\|_{L^q(T\times \Omega)}\leq & \|f_{\tilde{Y}}\|_{\infty} \left(\|\tilde X-\tilde Y\|_{L^q(T\times\Omega)}+\|\mathbb{W}_q^q(F_{\tilde{X}_{\cdot}},F_{\tilde{Y}_{\cdot}})\|_{L^1(T)}^{\frac 1q}\right).
    \end{align*}
    In particular,
    \begin{align*}
    \|U^X-U^Y\|_{L^q(T\times \Omega)}\leq  &2 \|f_{\tilde{Y}}\|_{\infty} \|\tilde X-\tilde Y\|_{L^q(T\times\Omega)}.
\end{align*}
\end{Theorem}

\begin{proof} Using the triangle inequality, we obtain
\begin{align*}
    \|U^Y-U^X\|_{L^q(T\times \Omega)}&\leq \|U^Y-F_{\tilde Y}(\tilde X)\|_{L^q(T\times \Omega)}+\|F_{\tilde Y}(\tilde X)-U^X\|_{L^q(T\times \Omega)} \nonumber \\
    &=:(1)+(2)
\end{align*}
Then for the first summand we have by the mean value inequality
\begin{align*}
    (1)^q=\|F_{\tilde Y}(\tilde Y)-F_{\tilde Y}(\tilde X)\|^q_{L^q(T\times \Omega)}= &\int_T \mathbb E [\vert F_{\tilde{Y}_t}(\tilde{Y}_t)-F_{\tilde{Y}_t}(\tilde{X}_t)\vert^q]dt \notag\\
    \leq & \|f_{\tilde Y}\|_{\infty}^q\int_T \mathbb E [\vert\tilde{Y}_t-\tilde{X}_t\vert^q]dt \notag\\
    = & \|f_{\tilde Y}\|_{\infty}^q\|\tilde Y- \tilde X\|_{L^q(T\times \Omega)}^q
\end{align*}
For the second summand we have again by the mean value theorem
\begin{align*}
    (2)^q= \|F_{\tilde Y}(F_{\tilde X}^{[-1]}(U^X_t))-U^X\|^q_{L^q(T\times \Omega)}= &\int_T\int_0^1 \vert F_{\tilde{Y}_t}(F_{\tilde{X}_t}^{[-1]}(u))-u\vert^qdudt\notag\\
    = &\int_T\int_0^1 \vert F_{\tilde{Y}_t}(F_{\tilde{X}_t}^{[-1]}(u))-F_{\tilde{Y}_t}(F_{\tilde{Y}_t}^{[-1]}(u))\vert^qdudt\notag\\
   \leq & \|f_{\tilde Y}\|_{\infty}^q \int_T\int_0^1 \vert F_{\tilde{X}_t}^{[-1]}(u)-F_{\tilde{Y}_t}^{[-1]}(u)\vert^qdudt\notag\\
   = &\|f_{\tilde Y}\|_{\infty}^q\|\mathbb{W}_q^q(F_{\tilde{X}_{\cdot}},F_{\tilde{Y}_{\cdot}})\|_{L^1(T)}.
\end{align*}
Moreover, since $\|\mathbb{W}_q^q(F_{\tilde{X}_t},F_{\tilde{Y}_t})\|_{L^1(T)}=\mathbb W_q^q(\tilde X, F_{\tilde Y}(U^X))$, which by Remark 4.5 can be estimated as, $$\mathbb W_q^q(\tilde X, F_{\tilde Y}(U^X))\leq \mathbb W_q^q(\tilde X, F_{\tilde Y}(U^Y))=\mathbb W_q^q(\tilde X,\tilde Y)\leq \mathbb E[\|\tilde X-\tilde Y\|_{L^q(T)}^q]=\|\tilde X-\tilde Y\|_{L^q(T\times \Omega)}^q,$$
also the second assertion follows.
\end{proof}
\begin{Example}\label{Ex: Gaussian copula approximation}
Assume that $U^Y$ is an elliptical copula corresponding to an elliptical random variable $\tilde{Y}$ in $L^2(T)$, that is
\begin{equation}\label{elliptical Random variable}
   \tilde Y=SV
\end{equation}
for some positive, real-valued random variable $S$ with finite second moment and a Gaussian process
$V\sim \mathcal N (0,C)$, independent of $S$ (see \cite{Boente2014} for the exact description and the relation to finite-dimensional elliptical distributions).
First, observe that without loss of generality we can assume $V_t\sim \mathcal N(0,1)$ 
since the process  $(\frac{\tilde{Y}_t}{\mathbb E [|\tilde{Y}_t|^2]})_{t\in T}$ has by Lemma  \ref{Lem: Same copulas means similarly ordered} the same copula as $\tilde Y$. 
If $S$ has finite inverse moment, then $\tilde{Y}_t$ has for each $t$ a bounded density, since by the formula for the density of two independent products, we get
\begin{equation*}
    f_{\tilde{Y}_t}(z)=\int_{-\infty}^{\infty} f_S(x)f_{V_t}(\frac zx)\frac 1{|x|}dx\leq \frac{1}{\sqrt{2\pi}}\int_{0}^{\infty}\frac{f_S(x)}{x}dx=\frac{\mathbb E [S^{-1}]}{\sqrt{2\pi}}.
\end{equation*}
Hence, for any other copula process $U^{\tilde X}$ corresponding to another process $\tilde{X}$ in $L^2(\Omega\times T)$, we have by Theorem \ref{Thm: Difference of copulas that stem from other processes} that
\begin{equation*}
    \|U^{\tilde{Y}}-U^{\tilde X}\|_{L^2(\Omega\times T)}\leq \sqrt{\frac{2}{\pi}}\mathbb E [S^{-1}] \|\tilde Y-\tilde X\|_{L^2(\Omega\times T)}.
\end{equation*}
\end{Example}


Copulas may be suitable to capture tail behaviour in functional data. This can be seen by the following example, combining the last ones.
\begin{Example}[Approximating Pareto marginals on an elliptical copula]
Assume that a process $Y$ given by
$Y_t:=F_t^{[-1]}(U_t)$ where $U$ is a copula process corresponding to an elliptical process $\tilde Y$ given by \eqref{elliptical Random variable} and the marginals $F_Y$ are regularly varying as in Example \ref{Ex: Regularly varying marginals}. More specifically we can take $F_Y$ to follow Pareto marginals, that is 
\begin{align*}
    l_t(x)= \begin{cases} \alpha_t x_{\text{min}}^{\alpha_t} & x\geq x_{min}\\
    0 & x\leq x_{min}
    \end{cases}
\end{align*}
for some constant $x_{min}>0$, such that $f_t$ are the densities of a Pareto distribution $Par(x_{min},\alpha_t)$, where $t\mapsto \alpha_t$ is assumed to be continuous. 
Assume now $\alpha_t>2+\gamma$ for some $\gamma>0$. Then $Y$ takes values in  $\mathcal{L}^2(T)$.

Consider a situation in which we can approximate the marginal function $F_{Y}$ by another marginal function $F_{n}$ (for example by empirical cumulative distribution functions).  
The underlying elliptical process is in the $L^2(T)$-norm best approximated over all processes with $n$ dimensional spectral decomposition by the projection
\begin{equation*}
    \tilde Y^n=\sum_{i=1}^{n} Z_i e_i 
\end{equation*}
of the
first $n$ principal components in the corresponding Karhunen–Lo\`eve expansion 
\begin{equation*}
    \tilde Y=\sum_{i=1}^{\infty} Z_i e_i .
\end{equation*}
Here $(e_i)_{i\in\mathbb N}$ is an  an orthonormal basis of eigenvectors of the covariance operator of $\tilde Y$, where the corresponding eigenvalues $(\lambda_i)_{i\in \mathbb N}$ are ordered decreasingly (see for instance Theorem (1.5) in \cite{Bosq2012} for a proof and \cite{Boente2014} for more optimality properties of the principal components for elliptical processes). 
Then
\begin{equation}\label{Approximation of principal components}
    \|\tilde{Y}_n-\tilde Y\|_{L^2(\Omega\times T)}^2=\sum_{i=n+1}^{\infty}\lambda_i.
\end{equation}
Assume that $U^n$ is the path copula process underlying $\tilde Y^n$.  
Let us investigate how well $Y^n=F_{n}^{[-1]}(U^n)$ approximates $Y$.
 We can choose $p=1$, $\epsilon=1$, $q=2$, $\beta=\frac 23$ and hence $\rho= \frac 13$, $m_t=x_{min}$ (using the notation of Theorem \ref{Thm:Robustness inequalty}) and $x_0^t=0$. Thus, by \eqref{Copula robustness estimate} 
 \begin{align}\label{Pareto first estimate}
        \mathbb E [\|Y-Y^n\|_{L^1(T)}]\leq &\|\mathbb{W}_1(F_{Y_{\cdot}},F_{n_{\cdot}})\|_{L^1(T)}+ K\|U^Y-U^n\|_{L^2(\Omega\times T)}^{\frac 13}.
     \end{align}
     Since in this case 
    \begin{align*}
        \mathbb E \left[f_t^{-\frac 23}(Y_t)\right]= \int_{x_{min}}^{\infty}\left(x^{-(\alpha_t+1)}\alpha_t x_{\text{min}}^{\alpha_t}\right)^{\frac 13}dx=\frac{3\alpha_t^{\frac 13} }{\alpha_t-2} x_{min}^{\frac 23}\leq \frac 3{\gamma} \alpha_t^{\frac 13} x_{\text{min}}^{\frac 23}
    \end{align*}
    and
    \begin{align*}
       \|Y\|_{L^{p+\epsilon}(\Omega\times T)}^{p+\epsilon}=\int_T \mathbb E [Y_t^2]dt=\int_T \frac{\alpha_t x_{min}^2}{\alpha_t-2}dt\leq \frac{x_{min}^2}{\gamma}\int_T \alpha_tdt
    \end{align*}
    we get by \eqref{Constant K}
    \begin{align}\label{Pareto constant K}
        K=  \left(\frac 6{\gamma}x_{\text{min}}^{\frac 23}\int_T \alpha_t^{\frac 13}  dt\right)^{\frac 12} \left( \frac{2}{\gamma}x_{min}^2\int_T \alpha_tdt\right)^{\frac 23}.
    \end{align}
    Thus, using also Example \ref{Ex: Gaussian copula approximation} and combining \eqref{Approximation of principal components}, \eqref{Pareto first estimate} and \eqref{Pareto constant K} we obtain
    \begin{align*}
       & \|Y-Y^n\|_{L^1(T\times \Omega)}\\
        \leq  &\|\mathbb{W}_1(F_{Y_{\cdot}},F_{n_{\cdot}})\|_{L^1(T)}+K \left(\sqrt{\frac{2}{\pi}}\mathbb E [S^{-1}]\right)^{\frac 13}\left(\sum_{i=n+1}^{\infty}\lambda_i\right)^{\frac 16}
    \end{align*}
and therefore, the convergence rate is $\frac 13$ of the convergence rate of the principal components and the rate of convergence induced by the Wasserstein distance of the marginals, which depends on the respective approximation technique for the marginals.
\end{Example}
The example above depicts a situation that can be investigated from the point of view of functional data analysis. 
As in finite dimensions the estimation of the underlying covariance matrix corresponding to an underlying elliptical copula is conducted via infering on the rank correlation, such as Kendall's $\tau$ or Spearman's $\rho$. A generalization of these objects and their estimation would have to be generalised to the functional setting. Nevertheless, this is a logical next step and is hence an appealing strand that is left for future research.

\subsection*{Acknowledgements}
This research was funded within the project STORM: Stochastics for Time-Space Risk Models, from the Research Council of Norway (RCN). Project number: 274410.

\bibliographystyle{amsplain}
\addcontentsline{toc}{section}{\refname}\bibliography{Bibliography}

\appendix

\section{Copulas in Finite Dimensions}\label{sec: Copulas in finite dimensions}
We begin with the definition of finite-dimensional cumulative distribution functions. 
\begin{Definition}
Let $J$ be a finite set.
A function $F:(-\infty,\infty)^J\to [0,1]$ is a cumulative distribution function on $\mathbb{R}^J$, if
\begin{itemize}
\item[(a)] $\lim_{x_j\to\infty,j\in J}F((x_j)_{j\in J})=1$;
\item[(b)] For each $j_0 \in J$ $\lim_{x_{j_0}\to -\infty}F((x_j)_{j\in J})=0$;
\item[(c)] For each $i\in J$ the function $t\to F((x_j)_{j\in J\setminus \lbrace i \rbrace},(t)_{j=i})$ is right-continuous for each $(x_j)_{j\in J\setminus \lbrace i\rbrace} \in \mathbb{R}^{J\setminus \lbrace i\rbrace}$; 
\item[(d)] The $F$-volume of a multivariate interval $[a,b]:=\times_{j\in J}[a_j,b_j]$ 
\begin{equation}
V_F([a,b]):=\sum_{v\in \prod_{j\in J}\lbrace a_j,b_j\rbrace} sign(v) F(v)
\end{equation}
is nonnegative,  that is $V_F([a,b])\geq 0$ for all $[a,b]:=\prod_{j\in J}[a_j,b_j]\subset \mathbb{R}^J$. (Recall that the function $sign:\prod_{j\in J}\lbrace a_j,b_j\rbrace\to\lbrace -1,1\rbrace$ is given by $sign(v)=(-1)^{N(v)}$, where $N(v)=\#\lbrace j\in J: v_j=a_j\rbrace$.)
\end{itemize}
\end{Definition}
\begin{Definition}\label{D: Finite dimensional Copula}
Let $J$ be an arbitrary finite set. A copula on $\mathbb{R}^J$ is a cumulative distribution function $C:[0,1]^J \to [0,1]$ with uniform marginal distributions, i.e. for all $u\in [0,1]$ we have
$$C((1)_{j\in J\setminus \lbrace i \rbrace},(u)_{j=i})=u.$$
\end{Definition}
Equivalently, each copula $C$ can be uniquely identified with a probability measure $\mu_C$, with cumulative distribution function $C$. 

For the theory of copulas, the most important result is Sklar's Theorem:
\begin{Theorem}[Sklar's Theorem in finite dimensions]\label{T: Finite Sklar}
Let $J$ be an arbitrary finite set.
Let $F$ be a cumulative distribution function on $\mathbb{R}^J$ with marginal one-dimensional cumulative distribution functions $F_j$ for each $j\in J$.
 Then there exists a copula $C$ on $\mathbb{R}^J$, such that for all $(x_j)_{j\in J}\in (-\infty,\infty)^J$ we have
\begin{equation}\label{E:Finite Copula Property}
F((x_j)_{j\in J})=C((F_j((x_j))_{j\in J}).
\end{equation}
If the marginals $F_j$ are continuous for each $j\in J$, $C$ is unique.
If conversely $C$ is a copula on $\mathbb{R}^J$ and $F_j$ are one-dimensional cumulative distribution functions for each $j\in J$, then $F$ defined by \eqref{E:Finite Copula Property} is a cumulative distribution function on $\mathbb{R}^J$ with marginals $F_j$ for each $j\in J$.
\end{Theorem}
\begin{proof}
See for example Nelsen \cite{Nelsen2006}. 
\end{proof}
We use the construction of copulas by distributional transforms from \cite{Ruschendorf2009}:
\begin{Theorem}\label{T: Rüschendorf Transform}
Let $J$ be a finite set, $F=F_{J}$ be a cumulative distribution function on $\mathbb{R}^J$ with marginals $F_j,j\in J$. Let $X=(X_j)_{j\in J}$ be a random vector with law $F$. 
Let $U$ be uniformly distributed on $[0,1]$ and independent of $X$.
Then a copula of $F$ is given by the cumulative distribution function corresponding to the random vector $(U_1,...,U_d)$, defined by
\begin{equation}\label{Distributional transform}
U_i=F_{X_i}(X_i-)+U(F(X_i)-F(X_i-)).
\end{equation}
\end{Theorem}
\begin{proof}
See \cite{Ruschendorf2009}.
\end{proof}
\end{document}